\documentclass[11pt,twoside]{amsart}
\usepackage{mathrsfs}
\usepackage{txfonts}
\usepackage{bbm}
\usepackage{amsfonts}
\usepackage{amsmath}
\usepackage{esint}
\usepackage{amssymb}
\usepackage{wasysym}
\usepackage{psfrag}
\usepackage{graphics,epsfig,amsmath}
\usepackage{comment}
\usepackage{enumerate}
\usepackage{fancyhdr}

\newtheorem{Theorem}{Theorem}[section]
\newtheorem{Lemma}[Theorem]{Lemma}
\newtheorem{Proposition}{Proposition}[section]
\newtheorem{Definition}{Definition}[section]
\newtheorem{remark}[Theorem]{Remark}

\usepackage{metalogo}
\usepackage[cmyk]{xcolor}
\selectfont
\usepackage{mathtools}
\mathtoolsset{showonlyrefs}
\renewcommand{\epsilon}{\varepsilon}

\title[Regularity for mixed operators]{ On some regularity properties \\ of mixed local and nonlocal \\
elliptic equations}

\author[X. Su]{Xifeng Su}
\address{School of Mathematical Sciences, Laboratory of Mathematics and Complex Systems (Ministry of Education)\\
Beijing Normal University,
No. 19, XinJieKouWai St., HaiDian District, Beijing 100875, P. R. China}
\email{xfsu@bnu.edu.cn, billy3492@gmail.com}

\author[E. Valdinoci]{Enrico Valdinoci}
\address{Department of Mathematics and Statistics, University of Western Australia, 35 Stirling Highway, WA 6009 Crawley, Australia}
\email{enrico.valdinoci@uwa.edu.au}
	
\author[Y. Wei]{Yuanhong Wei}
\address{School of Mathematics, Jilin University, No. 2699, Qianjin St., Changchun 130012, P. R.  China}
\email{weiyuanhong@jlu.edu.cn}

\author[J. Zhang]{Jiwen Zhang}
\address{School of Mathematical Sciences, Laboratory of Mathematics and Complex Systems (Ministry of Education)\\
	Beijing Normal University,
	No. 19, XinJieKouWai St., HaiDian District, Beijing 100875, P. R. China}
\email{jwzhang826@mail.bnu.edu.cn,jiwen.zhang@uwa.edu.au}


\begin{document}
\maketitle
\begin{abstract}
This article is concerned with ``up to $C^{2, \alpha}$-regularity results'' about a mixed local-nonlocal nonlinear elliptic equation which is driven by the superposition of Laplacian and fractional Laplacian operators. 

First of all, an estimate on the $L^\infty$ norm of weak solutions is established for more general cases than the ones present in the literature, including here critical nonlinearities.

 We then prove the interior $C^{1,\alpha}$-regularity and the $C^{1,\alpha}$-regularity up to the boundary of weak solutions, which extends previous results by the authors [X. Su, E. Valdinoci, Y. Wei and J. Zhang, Math. Z. (2022)], where the nonlinearities considered were of subcritical type. 

In addition, we establish the interior $C^{2,\alpha}$-regularity of solutions for all $s\in(0,1)$ and the $C^{2,\alpha}$-regularity up to the boundary for all $s\in(0,\frac{1}{2})$, with sharp regularity exponents.

For further perusal, we also include a strong maximum principle and some properties about the principal eigenvalue.
\end{abstract}

\tableofcontents

\section{Introduction}
The present paper is concerned with some regularity results about the following mixed local and nonlocal elliptic equation
\begin{equation}\label{Maineq}
\left\{%
\begin{array}{ll}
    - \Delta u +  (-\Delta)^s u = g(x,u) & \hbox{in $\Omega$,} \\
    u=0 &  \hbox{in $\mathbb{R}^n\backslash\Omega$,} \\
\end{array}%
\right.
\end{equation}
where $s\in (0,1)$ and~$\Omega$ is a bounded open set in  $\mathbb{R}^n$. 

As customary, the fractional Laplacian $ (-\Delta )^{s} $ is defined as
\begin{equation}
(-\Delta )^{s} u(x)=c_{n,s}\,{\text P.V.}\int_{\mathbb{R}^{n}} \frac{u(x)-u(y)}{|x-y|^{n+2s}}\, dy,
\end{equation}
 in which $ c_{n,s}>0 $ is a suitable normalization constant, whose explicit value does not play a role here, and~${\text P.V.}$ means that the integral is taken in the Cauchy principal value sense.


The superposition operator
\begin{equation}\label{OPL}
\mathcal{L}=-\Delta  + (-\Delta)^s,
\end{equation}
is of nonlocal type, due to the presence of the fractional Laplacian, hence
equation~\eqref{Maineq} is  a nonlocal elliptic equation in a bounded domain $\Omega$ with homogeneous
Dirichlet exterior data.

We also recall that
the fractional Laplacian is the generator of a L\'{e}vy flight, while the Laplacian is the generator of a Brownian motion. Based on this, equation~\eqref{Maineq} can naturally model the combined effect of Brownian motions and L\'{e}vy flights. As observed in~\cite{DV21}, these operators describe a biological species whose individuals diffuse either by a random walk or by a jump process. The analysis of different types of mixed operators motivated by biological questions has also been
carried out in \cite{MPV13, PV18}.

Elliptic equations driven by the above mixed operator have recently received a great attention from
different points of view, including viscosity solution theory \cite{JK05, BI08, BJK10, BCCI12}, existence and non-existence theory~\cite{AC21, BDVV22a, SVWZ}, eigenvalue problems~\cite{DPFR19, DPLV22},
shape optimization results~\cite{BDVV23},
symmetry properties of the solutions~\cite{BVDV21}, etc. 

In contrast with 
the cases in which the pure classical Laplacian and the pure fractional Laplacian
are separately taken into account, this kind of problems combines the classical setting and the features typical of nonlocal operators in a framework that is not scale-invariant. Hence,
 different scales will lead to different features of the mixed operator from the case of single classical Laplacian or single fractional Laplacian, thereby making either one of them prevail, or both somewhat coexist in the end.

{
The purpose of this article concentrates on establishing up to $C^{2,\alpha}$-regularity theory for mixed local-nonlocal elliptic equations, both in the interior and up to the boundary.

Our results are divided according to the following aspects:}.
\begin{itemize}
\item First of all, we extend the regularity result of \cite{SVWZ22} to some critical problems. More precisely, we prove the boundedness in $L^\infty$ of any weak solution of nonlinear problem~\eqref{Maineq}  with the critical Sobolev embedding exponent
\begin{equation}
	2^*=\begin{cases}
			\frac{2n}{n-2} &{\mbox{if }}n>2,\\
			\infty &{\mbox{if }}n\leqslant2.
		\end{cases}
	\end{equation}

{\item We then  deduce a $C^{1, \alpha}$-regularity result up to the boundary. This will actually be obtained from the $W^{2,p}$-regularity of weak solutions  of problem \eqref{Maineq}. Moreover, we discuss the interior  $C^{1,\alpha}$-regularity for all $\alpha\in(0,1)$. 

\item Next we  show the interior  $C^{2,\alpha}$-regularity for all $s\in(0,1)$. We develop a bespoke argument to get over a difficulty caused by the nonlinear terms. Our argument combines a truncation technique with the above  $C^{1, \alpha}$ estimates. We notice when $g$ depends only on $x$, the interior  $C^{2,\alpha}$-regularity of problem \eqref{Maineq}  is studied in \cite{BDVV22b}.
 Our result is thereby a nonlinear version of~\cite{BDVV22b}. 
However, even in the case of linear equations, our method is  different from \cite{BDVV22b} since the working space that we choose here is different.}
\item Finally, we show  the  $C^{2,\alpha}$-regularity up to the boundary for \eqref{Maineq} when~$s\in (0,\frac{1}{2})$ and discuss the case when ~$s\in (\frac{1}{2}, 1)$. 
A restriction on the fractional exponent for this type of regularity is somewhat unavoidable,
as shown by one-dimensional examples in which one picks a solution of~\eqref{Maineq} in~$\Omega:=(0,1)$ which behaves, near~$0^+$, as~$ax+O(x^2)$, for some constant~$a$.

Also, the nonlinear term brings some difficulties for the bootstrap argument, in view of technical regularity issues related to the composition of functions with low and high smoothness. \item As a byproduct of our main results, we also provide in the appendix some auxiliary results in a formulation which is ``ready for use'' in forthcoming papers, such as a strong maximum principle and a characterization of the principal eigenvalue.
\end{itemize}

Now we start presenting our main results. First off, in the strategy that we proposed above, the $L^\infty$-regularity theory goes as follows.
\begin{Theorem}[$L^\infty$-regularity]\label{th: regularity}
	Let $ u$  be   a weak solution\footnote{The precise definition of weak solution will be recalled in Definition~\ref{definition of weak solution}.} of \eqref{Maineq}.
	Assume that there exist $ c>0 $ and $q\in [1,2^*]$ such that
	\begin{equation}\label{eq: H 1}
	|g(x,t)|\leqslant c(1+|t|^{q-1}) \quad \text{ for a.e. } x\in  \Omega ,  t\in \mathbb{R}.
	\end{equation}Then, $ u\in L^{\infty}(\Omega)$. 
	
	More precisely, there exists a constant $C_0= C_0(c,n, q, \Omega)>0$ independent of $u$, such that
	\begin{equation*}
	\|u\|_{L^\infty(\Omega)}\leqslant C_0\left(1+\int_{\Omega}|u|^{2^*\beta_1}\,dx\right)^{\frac{1}{2^*(\beta_1-1)}},
	\end{equation*}where~$\beta_1:=\frac{2^*+1}{2}$. 
\end{Theorem}
 We point out that, the $L^{2^*\beta_1}$-norm of $u$ appearing in Theorem~\ref{th: regularity} is finite (as guaranteed by the forthcoming  Lemma~\ref{lemma u in L beta1}; in fact, as pointed out in~\eqref{eq:u in L beta1},
one can control the $L^{2^*\beta_1}$-norm by the $L^{2^*}$-norm,
hence by the energy of the solution). 

%

	\begin{remark}{\rm
In~\cite{SVWZ22}, we established
		the $L^\infty$ boundedness of   weak solution,  provided that the exponent of the  nonlinear term is controlled by
				\begin{equation}
				2^*_s
				=\begin{cases}\label{fractional}
					\frac{2n}{n-2s} &{\mbox{if }}n>2s,\\
					\infty &{\mbox{if }}n\leqslant2s.
				\end{cases}
		\end{equation} Here $2^*_s$ is the critical Sobolev embedding exponent for the space $H^s(\mathbb{R}^n)$ and characterizes the critical case for the operator $(-\Delta)^s$.
		
			From the above definitions of $2^*$ and $2_s^*$, one can easily see that $2^*\geqslant2^*_s$ (and the inequality is strict when $n>2s$). This suggests that, in several cases of interest,  the critical exponent is actually $2^*$ rather than $2^*_s$, due to the interpolation inequality, and correspondingly
Theorem~\ref{th: regularity} can be seen as a sharpening of~\cite[Theorem~1.1]{SVWZ22}.}
	\end{remark}

Based on the $L^\infty$-regularity theory, one can employ the $W^{2,p}$ theory of weak solutions  to deduce a global $ C^{1,\alpha} $-regularity  as follows.

\begin{Theorem}[$C^{1,\alpha}$-regularity up to the boundary]\label{th:C^{1,alpha}}
Let $ u$  be   a weak solution of~\eqref{Maineq}.
	Assume that $g$ satisfies \eqref{eq: H 1}  and 
	$ \partial\Omega $ is of class  $ C^{1,1} $. 
	
	Then, $ u\in C^{1,\alpha}(\overline\Omega)$ for any $\alpha\in(0,\min\left\{1,2-2s\right\})$.
%
%

More precisely, for~$s$ and~$\alpha$ as above, it holds that
\begin{equation}
	\|u\|_{C^{1,\alpha}(\overline{\Omega})}\leqslant C \left(1+\|u\|_{L^\infty(\Omega)}^{2^*}\right),
\end{equation}
where the positive constant $C$ depends only on $n,c,\alpha,\Omega,s$.
\end{Theorem}
We remark that Theorems~\ref{th: regularity} and~\ref{th:C^{1,alpha}} can be seen as a critical exponent version
of the $L^\infty$-regularity and of the $ C^{1,\alpha} $-regularity in \cite{SVWZ22}, respectively.

	We next present an interior regularity theory for weak solutions of the equation without specifying any external condition.
Namely, by combining a suitable truncation method and a covering argument, we obtain the interior $ C^{1,\alpha} $-regularity and the interior $ C^{2,\alpha} $-regularity, as formulated in the following two results: 
	
		\begin{Theorem}[Interior $ C^{1,\alpha} $-regularity]\label{th C^1,alpha interior without boundary condition}
	Let $ g(x,t)\in L^\infty_{\rm loc}(\Omega\times\mathbb{R}) $, $ u \in X^1$ 
	be  a bounded  weak solution of 
	\begin{equation}\label{1.5}
		- \Delta u +  (-\Delta)^s u = g(x,u) \qquad  \hbox{in $\Omega$},
	\end{equation}
where $X^1$ will be defined in \eqref{X_1}.  Let us denote 
	\begin{equation}\label{iu}
			\quad {\rm I_u}:=[-\|u\|_{L^\infty(\Omega)},\|u\|_{L^\infty(\Omega)}].
				\end{equation}
If~$V$ is an open domain with $\overline{V}\subset\Omega$,  and denote
		\[\rho:=\text{ dist}(V,\partial\Omega),\quad \text{  and } \quad  V_{\delta}:=\left\{x\in \Omega: \text{ dist}(x, V)<\delta\right\}. \]
Then, $u\in C^{1,\alpha}(\overline V)$ for any $\alpha\in (0,1)$.	

More specifically, 
	\begin{equation}\label{final inequality }
		\|u\|_{C^{1,\alpha}(\overline V)}\leqslant C\left(\|g\|_{L^\infty( V_{\frac{3\rho}{4}} \times{\rm I_u})}+ \|u\|_{L^\infty(\mathbb{R}^n)}\right)
	\end{equation}
	where  the constant $ C>0 $ depends on $ n,s,\alpha,V $.
\end{Theorem}

%
%
%
%


\begin{Theorem}[Interior $ C^{2,\alpha} $-regularity]\label{th C^2,alpha interior without boundary condition}
				Let $ g(x,t)\in C^{\alpha}_{\rm loc}(\Omega\times\mathbb{R}) $ for any  $\alpha\in(0,1)$,  $ u\in X^1$ 
	 be  a  bounded  weak solution of
	 \begin{equation}\label{1.6}
- \Delta u +  (-\Delta)^s u = g(x,u) \qquad  \hbox{in $\Omega$}.
	 \end{equation}
 Assume that $V$ is an open  domain with $\overline{V}\subset\Omega$. 
 
 Then, $u\in C^{2,\alpha}(\overline V)$.
 
 Furthermore,
	\begin{equation}\label{final inequality 2 }
	\|u\|_{C^{2,\alpha}(\overline V)}\leqslant C\left(\|u\|_{L^\infty(\mathbb{R}^n)}+\|g\|_{C^\alpha(\overline{V}_{\frac{3\rho}{4}} \times{\rm I_u})}\right) \left(1+\|g\|_{C^\alpha(\overline{V}_{\frac{3\rho}{4}} \times{\rm I_u})}\right),
	\end{equation}
	where  the constant $ C>0 $ depends on $ n,s,\alpha,V $. 
\end{Theorem}



Now we also consider the global $ C^{2,\alpha} $-regularity up to the boundary:

\begin{Theorem}[$ C^{2,\alpha} $-regularity up to the boundary]\label{theorem C^2 alpha global}
	Suppose that $ \partial\Omega $ is of class  $ C^{2,\alpha} $. Let $ s\in(0,\frac{1}{2}) $ and $ \alpha\in(0,1) $ be such that
		\begin{equation}\label{alpha if s in (0,1/2)}
		\alpha+2s\leqslant 1.
		\end{equation}
	If  $ u$ is a weak solution of  \eqref{Maineq} and $ g\in C^\alpha(\overline\Omega\times\mathbb{R}) $ satisfies \eqref{eq: H 1}, then $u\in C^{2,\alpha}(\overline\Omega).$ 	
	
	Moreover, we have that
\begin{equation}
	\|u\|_{ C^{2,\alpha}(\overline\Omega)}\leqslant C \left(\|g\|_{C^\alpha(\overline{\Omega}\times{\rm I_u})}\|u\|^{2^*}_{L^\infty(\Omega)}+\|g\|_{C^\alpha(\overline{\Omega}\times{\rm I_u})}+\|u\|_{L^\infty(\Omega)}\right)
\end{equation}
for some positive constant $C$ depending only on $n,s,c,\alpha,\Omega$.
\end{Theorem}

\begin{remark}{\rm
In \cite[Theorem 1.6]{BDVV22b}, a related~$ C^{2,\alpha} $-regularity result is proved for linear equation. This result relies on the smoothness of the linear term and is not directly applicable to the nonlinear equation \eqref{Maineq}. 

In our context, for the interior $ C^{2,\alpha} $-regularity, we will deal with several technical issues coming from the nonlinearity $ g(x,u(x))$ to guarantee that one can carry out the iterative process.
For the $ C^{2,\alpha} $-regularity up to the boundary, we also treat the mixed operator with $s\in (0,\frac{1}{2})$ as a perturbation of the classical Laplacian and apply the classical results of local problems to obtain the main result.
}\end{remark}

The rest of this paper is organized as follows.
In Section~\ref{sec:preliminary}, we first introduce several preliminaries and define the working space.
Then, the $L^{\infty}$-boundedness of weak solutions
is established in Section~\ref{sec:regularity}, 
while Section~\ref{sec: C^alpha-regularity} is devoted to the~$C^{1,\alpha}$-regularity up to the boundary and  the interior ~$C^{1,\alpha}$-regularity. 

Next, the interior $C^{2,\alpha}$-regularity and the $C^{2,\alpha}$-regularity up to the boundary are shown in Section~ \ref{sec: C^2- regularity}. 

In the appendix we collect some simple consequences of the main results, such as a
strong maximum principle and a variational characterization for the principal eigenvalue.
These results are not stated in their maximal generality, but rather in a form which is ready to use
in nonlinear analysis problems (see~\cite{SVWZ}).

\section{Preliminaries and additional results}\setcounter{equation}{0}\label{sec:preliminary}
Let $ s\in(0,1)$ be given and $ \Omega\subset \mathbb{R}^{n}$ be an open bounded set.  
We introduce our working space, which is slightly different from  e.g.~\cite{BDVV22b,SVWZ22}. 

Let us first define the semi-norm 
\[
[u]_s:= \bigg(\int\!\!\!\int\nolimits_{\mathbb{R}^{2n}}\frac{|u(x)-u(y)|^2}{|x-y|^{n+2s}}\, dxdy\bigg)^{\frac{1}{2}}.
\]

In terms of the mixed  operator $\mathcal{L}$, a ``natural" function space to consider is the following space
\begin{equation}\label{X_1}
	X^1:=
	\left\{\begin{array}{ll}
		u:\mathbb{R}^{n}\rightarrow \mathbb{R} \text{ is Lebesgue measurable:} \vspace{0.2cm} \\ ~ u|_{\Omega} \in H^{1}(\Omega), ~~ [u]_s \text{ is finite}\\
	\end{array}
	\right\}
\end{equation}
equipped with the norm
\begin{equation*}
	\bigg(\|u\|_{H^{1}(\Omega)}^{2}+[u]_s^2 \bigg)^{\frac{1}{2}} .
\end{equation*}
 We denote $H_0^1(\Omega)$ as the closure of $C_c^\infty(\Omega)$ with respect to the $H^1(\Omega)$ norm, and consider the following space\footnote{We prefer to use the function space $H_0^1(\Omega)$ rather than $H^1(\Omega)$, or~$H^1(\mathbb{R}^n)$ with zero data outside~$\Omega$,
to avoid trace theorems which may impose additional regularity assumptions on~$\Omega$. Notice indeed
that no regularity on~$\Omega$ is assumed in Theorem~\ref{th: regularity}.
See however the forthcoming Lemma~\ref{lemma equivalence of norms} addressing the equivalence of the norms
involved in each of these choices.} of functions
\begin{equation}\label{working space}
	X^1_0:=
	\left\{\begin{array}{ll}
	u:\mathbb{R}^{n}\rightarrow \mathbb{R} \text{ is Lebesgue measurable:} \vspace{0.2cm} \\ ~ u|_{\Omega} \in H^{1}_{0}(\Omega), ~~ [u]_s \text{ is finite}, ~ ~ u=0 \text{ a.e. in  } \mathbb{R}^n\backslash \Omega \\
	\end{array}
	\right\}
\end{equation}
equipped with the norm
\begin{equation*}
\bigg(\|u\|_{H^{1}_0(\Omega)}^{2}+[u]_s^2 \bigg)^{\frac{1}{2}} .
\end{equation*}
Due to the Poincar\'e inequality, the above norm in $ X_{0}^1 $ is also equivalent to
\begin{equation}
\|u\|_{X_0^1} :=\bigg(\|\nabla u\|^{2}_{L^2(\Omega)}+[u]_s^2\bigg)^{\frac{1}{2}}.  \label{Xnorm}
\end{equation}
Obviously, $X^1_{0} $ is a Hilbert space equipped with the inner product
\begin{equation*}
\left\langle u,v \right\rangle_{X_{0}^1}=\int_{\Omega} \nabla u\cdot  \nabla v\, dx+\int\!\!\!\int_{\mathbb{R}^{2n}}\frac{(u(x)-u(y))(v(x)-v(y))}{|x-y|^{n+2s}}\, dxdy, \quad \forall\ u,v \in X_{0}^1.
\end{equation*}
The following is the definition of weak solutions of \eqref{Maineq} in $X_{0}^1$: 
\begin{Definition}[Weak solution]\label{definition of weak solution}
	We say that $ u\in X_{0}^1 $ is a weak solution of the mixed elliptic equation \eqref{Maineq},  if $ u $ satisfies 
	\begin{equation}
	\begin{split} \label{weaksolution}
	\displaystyle\int_{\Omega}\left\langle \nabla u,\nabla \varphi \right\rangle +& \displaystyle\int\!\!\!\int\nolimits_{\mathbb{R}^{2n}}\frac{(u(x)-u(y))(\varphi(x)-\varphi(y))}{|x-y|^{n+2s}}\, dxdy\\
	&=\int_{\Omega}g(x,u(x))\varphi(x)\, dx, \quad \text{  for any }  \varphi \in X_{0}^1.
	\end{split}
	\end{equation}
 Here, for typographical simplicity, we take the factor $\frac{c_{n,s}}{2}$ equal to $1$ in front of the inner product in~$H^s(\mathbb{R}^n)$ between $u$ and $\varphi$ since its explicit value does not play a role here.
\end{Definition}

We investigate some key facts of $ X_0^1 $ in the following lemmata on equivalence of norms,  embedding and density results.

\begin{Lemma}\label{lemma equivalence of norms}
In the space $ X_{0}^1 $, we have the following equivalence of norms:
\[
||\cdot||_{X_0^1}\sim ||\cdot||_{H^1_0(\Omega)}. 
\]
\end{Lemma}
\begin{proof} 
	Let $ u\in H_0^1(\Omega) $ and $ u=0 $ a.e. in $ \mathbb{R}^n\backslash\Omega $.
	Using the change of variable $ z=y-x $, one has 
	\begin{equation}
	\begin{split}\label{H1 -X}
	&\int_{\Omega}\int_{\mathbb{R}^{n}\backslash\Omega}\frac{|u(x)|^2}{|x-y|^{n+2s}}\, dydx\\
	=&\int_{\Omega}\int_{\{x+z\notin\Omega\}\cap\{|z|\leqslant 1\}}\frac{|u(x)|^2}{|z|^{n+2s}}\, dzdx
	+\int_{\Omega}\int_{\{x+z\notin\Omega\}\cap\{|z|>1\}}\frac{|u(x)|^2}{|z|^{n+2s}}\, dzdx\\
	=&:A_1+A_2.
	\end{split}
	\end{equation}
	
	We observe that when~$x+z$ lies in the complement of~$\Omega$ we have
\begin{eqnarray*}&&
|u(x)|^2=|u(x+z)-u(x)|^2=\left| \int_0^1 \nabla u(x+tz)\cdot z\,dt\right|^2\le \int_0^1 |\nabla u(x+tz)|^2 |z|^2\,dt
\end{eqnarray*}
and therefore, making use of the change of variable~$y:=x+tz$ and the fact that~$u$ vanishes outside~$\Omega$,
we find that 
	\begin{equation}
	\begin{split}\label{A1}
	A_1\leqslant C_1(n,s)\|u\|_{H_0^1(\Omega)}^2.
	\end{split}
	\end{equation}
	
	Moreover,
	\begin{equation}\label{A2}
	A_2\leqslant C_2(n,s)\|u\|_{L^2(\Omega)}^2.
	\end{equation}
	Also, it follows from \cite{DPV12} that 
	\begin{equation}\label{A3}
	\int\!\!\!\int\nolimits_{\Omega\times\Omega}\frac{|u(x)-u(y)|^2}{|x-y|^{n+2s}}\, dxdy\leqslant C_3(n,s)\|u\|_{H_0^1(\Omega)}^2.
	\end{equation}
	Combining \eqref{A1}, \eqref{A2} and \eqref{A3}, we deduce that 
	\begin{equation*}
	[u]_s^2\leqslant C(n,s)\|u\|_{H_0^1(\Omega)}^2,
	\end{equation*}
	for some suitable constant $ C(n,s). $ 
	
	As a consequence of this, we obtain the equivalence  between the two norms $||\cdot||_{X_0^1}$ and $||\cdot||_{H^1_0(\Omega)}$.
\end{proof}

\begin{Lemma}\label{lemma: embedding continuous}
	The embedding $ X_{0}^1\hookrightarrow L^{2^*}(\Omega) $ is continuous.
\end{Lemma}
\begin{proof}
	Let $ u\in X_0^1$, invoking the Sobolev inequality, and we have 
	$$ \|u\|_{L^{2^*}(\Omega)}\leqslant C\|u\|_{H^1_0(\Omega)}\leqslant C\|u\|_{X_0^1} $$
for some suitable positive constant $ C$.
\end{proof}
\begin{Lemma}\label{lemma densely embedding}
	Let $\mathscr{C}_0^\infty:=\left\{u\in C(\mathbb{R}^n): u=0 \text{ in }\mathbb{R}^n\backslash \Omega, \;u|_{\Omega}\in C_0^\infty(\Omega) \right\} $ and denote by $\overline{\mathscr{C}_0^\infty}$ the closure of $ \mathscr{C}_0^\infty $ with respect to $X_0^1$-norm. Then, we have that~$ \overline{\mathscr{C}_0^\infty} = X_0^1 $.
\end{Lemma}
\begin{proof}
	It  is not difficult to see that  $ \overline{\mathscr{C}_0^\infty}\subset X^1_0 . $
	Besides, according to Lemma~\ref{lemma equivalence of norms}  and the fact $ C_0^\infty(\Omega) $ is dense in $ H_0^1(\Omega) $, one can deduce that $ X^1_0\subset\overline{\mathscr{C}_0^\infty}. $
\end{proof}
	
\section{$L^{\infty}$-regularity}\label{sec:regularity}\setcounter{equation}{0}

In this section, we show that the weak solutions $u\in X_0^1$ are bounded
in $ L^{\infty}(\Omega) $. To prove this result, we will use the
Moser iteration method.

To begin with, we construct an auxiliary function $\varphi$ and provide several fundamental properties of $\varphi$, which are useful for the iterative procedure.
\begin{Lemma}\label{lemma  auxiliary function}
	Given $ \beta> 1 $ and $ T>0 $, we define 
	\begin{equation}\label{vsdvdsc}
		\varphi(t)=
		\begin{cases}
			-\beta T^{\beta-1}(t+T)+T^\beta, &\text{  if }t\leqslant -T\\
			|t|^\beta, &\text{  if } -T<t<T,\\
			\beta T^{\beta-1}(t-T)+T^\beta, &\text{  if }t\geqslant T.
		\end{cases}
	\end{equation}
	{Then $\varphi$ satisfies the following properties:
		\begin{itemize}
			\item[(a)] $\varphi\in C^1(\mathbb{R}, \mathbb{R})  $ is a convex function.
			\item[(b)] If $ u\in X_0^1 $, then  $ \varphi(u)$ and $\varphi(u)\varphi'(u)\in X_0^1$.
	\end{itemize}
\begin{proof}
 It is immediate to check that $\varphi\in C^1(\mathbb{R}, \mathbb{R})  $ is a convex function with
	\begin{equation*}
		\varphi'(t)=
		\begin{cases}
			-\beta T^{\beta-1}, &\text{  if }t< -T,\\
			-\beta(-t)^{\beta-1}, &\text{  if } -T\leqslant t\leqslant 0,\\
			\beta t^{\beta-1}, &\text{  if } 0\leqslant t\leqslant T,\\
			\beta T^{\beta-1}, &\text{  if } t>T,
		\end{cases}
	\end{equation*}
	and
	\begin{equation*}
		\varphi''(t)=\begin{cases}
			\beta(\beta-1)t^{\beta-2}, \qquad & 0<t<T,\\
			\beta(\beta-1)(-t)^{\beta-2}, \qquad & -T<t<0,\\
			0,& \text{ otherwise}
		\end{cases}
	\end{equation*}
	in the sense of distributions. This proves (a).
	
	We now observe that $\varphi(u)=0 $ a.e. in $ \mathbb{R}^n\backslash \Omega $. Furthermore, since $ \varphi $ is Lipschitz with Lipschitz constant $ L=\beta T^{\beta-1} $, one has that 
\begin{equation}\label{ svdsvsd}
	\begin{split}
		\|\varphi(u)\|^2_{L^2(\Omega)}&\leqslant \int_{\left\{|u|<T\right\}} T^{2\beta}\, dx+\int_{\left\{|u|>T\right\}} \left(\beta T^{\beta-1}(|u|-T)+T^\beta\right)^2\, dx\\
		&\leqslant 2\beta^2T^{2\beta-2}\|u\|^2_{L^2(\Omega)}+(4\beta^2+5)T^{2\beta}|\Omega|<\infty,
	\end{split}
\end{equation}
and	\begin{equation}
		\begin{split}
			\|\varphi(u)\|^2_{X_0^1}&=\|\nabla( \varphi(u))\|^2_{L^2(\Omega)}+\int\!\!\!\int\nolimits_{\mathbb{R}^{2n}}\frac{|\varphi(u(x))-\varphi(u(y))|^2}{|x-y|^{n+2s}}\, dxdy  \\
			&\leqslant L^2\bigg(\|\nabla u\|^2_{L^2(\Omega)}+\int\!\!\!\int\nolimits_{\mathbb{R}^{2n}}\frac{|u(x)-u(y)|^2}{|x-y|^{n+2s}}\, dxdy\bigg)<\infty.
			\label{4.1}
		\end{split}
	\end{equation}
Thus, we have that $\varphi(u)\in X_0^1$.

We now prove that $\varphi(u)\varphi'(u)\in X_0^1 $. Owing to Lemma~\ref{lemma equivalence of norms}, one has that 
\begin{equation}\label{vvrwdcs}
	\begin{split}
		&\|\varphi(u)\varphi'(u)\|^2_{X_0^1}\\
		=&C\|\varphi(u)\varphi'(u)\|^2_{H_0^1(\Omega)}\\
		\leqslant& C
		\left(\int_{\Omega}|\varphi(u)\varphi'(u)|^2\, dx+\int_{\Omega}|\nabla u|^2|\varphi'(u)|^4\, dx+\int_{\Omega}|\nabla u|^2|\varphi(u)\varphi''(u)|^2\, dx\right)\\
		\leqslant&  C
		\left(\beta^4T^{4\beta-4}\|u\|_{X_0^1}^2+(4\beta^2+5)\beta^2T^{4\beta-2}|\Omega|+\beta^2(\beta-1)^2T^{4\beta-4}\|u\|_{X_0^1}^2\right)<+\infty
	\end{split}
\end{equation}
for some constant $C$ depending on $n,s,\Omega$. This implies that $\varphi(u)\varphi'(u)\in X_0^1 $. The proof of (b) is thus complete.
\end{proof}
}
	
\end{Lemma}

\begin{Lemma}\label{lemma u in L beta1}
	 Under the assumptions of Theorem~\ref{th: regularity}, let $\beta_1$ is such that $ 2\beta_1-1=2^* $.
	 Then, $ u\in L^{2^*\beta_1}(\Omega)$.
\end{Lemma}
\begin{proof}
	Let  $\varphi$ be given in Lemma~\ref{lemma  auxiliary function}. We first claim that
	\begin{equation}\label{dvsdv}
		\left(\varphi(u(x))\varphi'(u(x))-\varphi(u(y))\varphi'(u(y))\right)\left(u(x)-u(y)\right)
		\geqslant|\varphi(u(x))-\varphi(u(y))|^2
	\end{equation}
	for  every $u\in X_0^1$ and $x,y\in\mathbb{R}^n$.
	
	Indeed,	since  $\varphi$ is convex, we have that
	\begin{equation}\label{4.2}
		\varphi(t_1)-\varphi(t_2)\leqslant \varphi'(t_1)\left(t_1-t_2\right) \qquad \forall  t_1,t_2\in\mathbb{R}.
	\end{equation}
By combining this with the  non-negativity of $\varphi$, we derive that, for  every $u\in X_0^1$ and $x,y\in\mathbb{R}^n$, 
	\begin{equation}\label{1}
		\varphi(u(x))\varphi'(u(x))\left(u(x)-u(y)\right)\geqslant \varphi(u(x))\left(\varphi(u(x))-\varphi(u(y))\right)
	\end{equation}
and	 	\begin{equation}\label{2}
	\varphi(u(y))\varphi'(u(y))\left(u(y)-u(x)\right)\geqslant \varphi(u(y))\left(\varphi(u(y))-\varphi(u(x))\right).
	\end{equation}
   Then, by adding~\eqref{1} and~\eqref{2} together, we obtain the claim in~\eqref{dvsdv}.
   
   Now we plug the test function $\varphi(u)\varphi'(u)$ into the
   equation~\eqref{Maineq}, one has that
   \begin{equation}\label{integration of the main equation}
   	-\int_{\mathbb{R}^{n}}(\varphi(u)\varphi'(u))\Delta u\,dx+\int_{\mathbb{R}^{n}}\varphi(u)\,\varphi'(u)\,(-\Delta)^su\, dx=\int_{\mathbb{R}^{n}}\varphi(u)\varphi'(u)g(x,u)\, dx.
   \end{equation}
   The first term on the left side of \eqref{integration of the main equation} can be rewritten as
   \begin{equation}\label{rewriting of the first term}
   	\begin{split}
   		&-\int_{\Omega} (\varphi(u)\varphi'(u))\Delta u\, dx= \int_{\Omega}\nabla u\cdot \nabla(\varphi(u)\varphi'(u))\, dx\\
   		=&\int_{\Omega}|\nabla u|^2|\varphi'(u)|^2\, dx+\int_{\Omega}|\nabla u|^2\varphi(u)\,\varphi''(u)\,  dx.
   	\end{split}
   \end{equation}
   Owing to the convexity and non-negativity of~$\varphi$, we see that  
   \begin{equation}\label{vdssd}
   	\int_{\Omega}|\nabla u|^2\varphi(u)\,\varphi''(u)\,  dx\geqslant 0.
   \end{equation} 
   As for the second term on the left side of \eqref{integration of the main equation}, utilizing~\eqref{dvsdv}, one has that  
   \begin{equation}
   	\begin{split}\label{3.2BIS}
   		&\int_{\mathbb{R}^{n}}\varphi(u)\,\varphi'(u)\,(-\Delta)^su\, dx\\
   		&\quad =\frac{1}{2}\int\!\!\!\int\nolimits_{\mathbb{R}^{2n}}\frac{\left(\varphi(u(x))\varphi'(u(x))-\varphi(u(y))\varphi'(u(y))\right)\left(u(x)-u(y)\right)}{|x-y|^{n+2s}}\, dxdy\geqslant 0.
   	\end{split}
   \end{equation}
   From this, by combining~\eqref{integration of the main equation},~\eqref{rewriting of the first term} and \eqref{vdssd}, one derives that
   \begin{align}
   	\int_{\Omega}|\nabla u|^2|\varphi'(u)|^2\, dx
   	&\leqslant\int_{\mathbb{R}^n}\varphi(u)\,\varphi'(u)\,g(x,u)\, dx. \label{4.5}
\end{align}
Hence, recalling~\eqref{eq: H 1} and utilizing the Sobolev inequality, we obtain
\begin{align*}
\|\varphi(u)\|^2_{L^{2^*}(\Omega)}&\leqslant C \|\nabla (\varphi(u))\|_{L^2(\Omega)}^2= C\int_{\Omega}|\varphi'(u) \nabla u|^2\,dx\\
& \leqslant C\int_{\mathbb{R}^n}\varphi(u)\varphi'(u)g(x,u)\, dx
\leqslant C \int_{\mathbb{R}^n}\varphi(u)|\varphi'(u)|(1+|u|^{2^*-1})\, dx\\
& =C\bigg(\int_{\mathbb{R}^n}\varphi(u)|\varphi'(u)|\,dx +\int_{\mathbb{R}^n}\varphi(u)|\varphi'(u)||u|^{2^*-1}\, dx \bigg),
\end{align*}
where the constant $ C $ depends only on $ n,\Omega $ and $ c $.
	
	Employing the estimates \begin{equation}\label{bjdsvbjds}
		\varphi(u)\leqslant |u|^\beta ,\qquad |\varphi'(u)|\leqslant \beta |u|^{\beta-1} \qquad{\mbox{and}}\qquad |u\varphi'(u)|\leqslant \beta \varphi(u),
	\end{equation} we see that
	\begin{equation}\label{main estimate}
	\bigg(\int_{\Omega}(\varphi(u))^{2^*}\,dx\bigg)^{2/2^*}\leqslant C\beta\bigg(\int_{\Omega}|u|^{2\beta-1}\,dx +\int_{\Omega}(\varphi(u))^2 |u|^{2^*-2}\, dx \bigg).
	\end{equation}
	Notice that $ C $ is a positive constant that does not depend on $ \beta $ and  that the last integral on the right-hand-side of the inequality \eqref{main estimate} is well defined for every~$ T >0 $. 
	
	Furthermore, for any given $ R>0 $, we apply the H\"older inequality  and obtain that
	\begin{equation}
	\begin{split}
	&\int_{\Omega}(\varphi(u))^2 |u|^{2^*-2}\, dx=\int_{\left\{|u|\leqslant R\right\}}(\varphi(u))^2 |u|^{2^*-2}\, dx +\int_{\left\{|u|>R\right\}} (\varphi(u))^2 |u|^{2^*-2}\, dx \\
	\leqslant& \int_{\left\{|u|\leqslant R\right\}}\frac{(\varphi(u))^2}{|u|} R^{2^*-1}\, dx+\bigg(\int_{\Omega}(\varphi(u))^{2^*}\, dx\bigg)^{2/2^*}\bigg(\int_{\left\{|u|>R\right\}}|u|^{2^*}\, dx\bigg)^{\frac{2^*-2}{2^*}}.
	\end{split}\label{estimate of phi u}
	\end{equation}
	By the Monotone Convergence Theorem, one can choose $R>0$ large enough such that
	\begin{equation}\label{eq:C beta}
	\bigg(\int_{\left\{|u|>R\right\}}|u|^{2^*}\,dx\bigg)^{\frac{2^*-2}{2^*}}\leqslant \frac{1}{2C\beta_1},
	\end{equation}
	where $ C $ is the constant in \eqref{main estimate}. 
	
	Therefore,  choosing $\beta=\beta_1$ in~\eqref{main estimate}, one thus can reabsorb the last term in \eqref{estimate of phi u} into the left hand side of \eqref{main estimate} to get
	\begin{equation*}
	\bigg(\int_{\Omega}(\varphi(u))^{2^*}\,dx\bigg)^{2/2^*}\leqslant 2C\beta_1\bigg(\int_{\Omega}|u|^{2^*}\, dx+ R^{2^*-1}\int_{\Omega}\frac{(\varphi(u))^2}{|u|}\, dx \bigg).
	\end{equation*}
	Now, using that $\varphi(u)\leqslant |u|^{\beta_1}$, we conclude that the terms in the right hand side of the above inequality is bounded independently of $T$.
	
	Sending $ T\rightarrow +\infty $, we obtain that
	\begin{equation}\label{eq:u in L beta1}
	\bigg(\int_{\Omega}|u|^{2^*\beta_1}\,dx\bigg)^{2/2^*}\leqslant 2C\beta_1\bigg(\int_{\Omega}|u|^{2^*}\, dx+ R^{2^*-1}\int_{\Omega}|u|^{2^*}\, dx \bigg)<+\infty,
	\end{equation}as desired.
\end{proof}

Now we complete the proof of the $L^\infty$-regularity result:

 \begin{proof}[Proof of Theorem~\ref{th: regularity}]
 	For $ m\geqslant 1, $ we define $ \beta_{m+1} $ such that
 	\begin{equation*}
 	2\beta_{m+1}+2^*-2=2^*\beta_{m}.
 	\end{equation*}
 	As in \cite{SVWZ22}, we infer that
 	\begin{equation*}
 	\bigg(1+\int_{\Omega}|u|^{2^*\beta_{m+1}}\,dx\bigg)^{\frac{1}{2^*(\beta_{m+1}-1)}}\leqslant (C\beta_{m+1})^{\frac{1}{2(\beta_{m+1}-1)}}\bigg(1+\int_{\Omega} |u|^{2^*\beta_{m}}\, dx\bigg)^{\frac{1}{2^*(\beta_{m}-1)}}.
 	\end{equation*}
 	We now define $ C_{m+1}:=C\beta_{m+1} $ and
 	\begin{equation*}
 	A_m:=\bigg(1+\int_{\Omega} |u|^{2^*\beta_{m}}\, dx\bigg)^{\frac{1}{2^*(\beta_{m}-1)}}.
 	\end{equation*}
 	In particular, note that $A_1= \bigg(1+\int_{\Omega} |u|^{2^*\beta_{1}}\, dx\bigg)^{\frac{1}{2^*(\beta_{1}-1)}}$ is bounded by \eqref{eq:u in L beta1}.
 	
 It is not difficult to verify that there exists a constant $ C_0>0 $ independent of $m$, such that
 	\begin{equation}\label{DRI}
 	A_{m+1}\leqslant \prod_{k=2}^{m+1}C_k^{\frac{1}{2(\beta_{k}-1)}}A_1\leqslant C_0 A_1.
 	\end{equation}
 	We stress that once~\eqref{DRI} is established,
 	then, by the H\"older inequality, we conclude that $ u\in L^p(\Omega)$, for every~$ p\in[1,+\infty)$. Furthermore,
 	a limit argument implies that
 	\begin{equation*}
 	\|u\|_{L^\infty(\Omega)}\leqslant C_0 {A_1}<+\infty,
 	\end{equation*}
 	 where $ C_0 $ depends only on $ n,q,\Omega,c $, which finishes the proof of Theorem~\ref{th: regularity}.
 \end{proof}

Moreover, based on the proof of Theorem~\ref{th: regularity}, we can deduce a bound of the $L^\infty$-norm of a weak solution $u$ by a constant which is independent of $u$ itself when the nonlinearity $g$ is strictly sublinear growth in $u$, as described in the following result.
	
	\begin{Proposition}\label{th: regularity for linearity}
		Let $ u$  be   a weak solution of \eqref{Maineq}.
		Assume that there exist $ \Lambda>0 $ and $q\in [1,2)$ such that
		\begin{equation}\label{eq: H 11}
			|g(x,t)|\leqslant \Lambda(1+|t|^{q-1}) \quad \text{ for a.e. } x\in  \Omega ,  t\in \mathbb{R}.
		\end{equation}Then, $ u\in L^{\infty}(\Omega)$. 
		
		More precisely, there exists a constant $C_0= C_0(\Lambda,n,s, q, \Omega)>0$ independent of $u$, such that
		\begin{equation}\label{vrg}
			\|u\|_{L^\infty(\Omega)}\leqslant C_0.
		\end{equation}
\end{Proposition}

We remark that the uniform estimate of Proposition~\ref{th: regularity for linearity}
does not hold when~$q\ge2$: as a counterexample, one can take the linear case in which $
g(x,t):=\lambda t$, with~$ \lambda$ an eigenvalue of the mixed operator~$\mathcal{L}$.
In this situation, we have that~$ |g(x,t)|\le \lambda(1+|t|)\le 2\lambda(1+|t|^{q-1})$ as long as~$q\ge2$; but, in this setting, if~$ u$ is a solution, then so is~$ Mu $ for every~$ M>0$, hence the $L^\infty$-norm of the solution cannot be bounded by a quantity independent on~$u$. This example is interesting because it shows that the range of exponents~$q$ in
Proposition~\ref{th: regularity for linearity} is optimal.

{\begin{proof}[Proof of Proposition~\ref{th: regularity for linearity}]
		We recall that 
		\begin{equation}\label{bdsgds}
			\|u\|^2_{L^{2^*_s}(\Omega)}\leqslant S [u]_s^2
		\end{equation}
		where  $2^*_s$ is the critical Sobolev embedding exponent given in~\eqref{fractional} and $S$ is a constant depending only on $n $ and $s$.
		Let  $\varphi$ be given in Lemma~\ref{lemma  auxiliary function}. Recalling ~\eqref{dvsdv} and~\eqref{3.2BIS}, employing~\eqref{bdsgds},  we see that 
	 \begin{equation}
	 	\begin{split}\label{vdsvds}
		&\bigg(\int_{\Omega}(\varphi(u))^{2_s^*}\,dx\bigg)^{2/2_s^*}\\
		\leqslant&
		S\bigg(\int\!\!\!\int\nolimits_{\mathbb{R}^{2n}}\frac{|\varphi(u(x))-\varphi(u(y))|^2}{|x-y|^{n+2s}}\, dxdy\bigg)\\
	\leqslant	& S\left(\int\!\!\!\int\nolimits_{\mathbb{R}^{2n}}\frac{\left(\varphi(u(x))\varphi'(u(x))-\varphi(u(y))\varphi'(u(y))\right)\left(u(x)-u(y)\right)}{|x-y|^{n+2s}}\, dxdy\right)\\
		=&2S \int_{\mathbb{R}^{n}}\varphi(u)\,\varphi'(u)\,(-\Delta)^su\, dx. 
	\end{split}
	\end{equation}
Recalling~\eqref{integration of the main equation}, we observe that
\begin{equation}\label{bfsb}
	\int_{\mathbb{R}^{n}}\varphi(u)\,\varphi'(u)\,(-\Delta)^su\, dx=\int_{\mathbb{R}^{n}}\varphi(u)\varphi'(u)g(x,u)\, dx+\int_{\mathbb{R}^{n}}(\varphi(u)\varphi'(u))\Delta u\,dx.
\end{equation}
From~\eqref{rewriting of the first term} and~\eqref{vdssd}, it follows that the second term on the right side of \eqref{bfsb} is non-positive, 
 one thus has that 
\begin{equation}
	\begin{split}
		\bigg(\int_{\Omega}(\varphi(u))^{2_s^*}\,dx\bigg)^{2/2_s^*}&\leqslant 2S\int_{\mathbb{R}^n}\varphi(u)\,\varphi'(u)\,g(x,u)\, dx.
	\end{split}
\end{equation}
Moreover, by combining~\eqref{bjdsvbjds} with~\eqref{eq: H 11}, one derives that
\begin{equation}\label{main estimate1}
	\begin{split}
		\bigg(\int_{\Omega}(\varphi(u))^{2_s^*}\,dx\bigg)^{2/2_s^*}
		&\leqslant 2S\Lambda\int_{\mathbb{R}^n}\varphi(u)\,\varphi'(u)\left(1+|u|\right)\, dx\\
		&\leqslant  C\beta\left(\int_{\Omega}|u|^{2\beta-1}\,dx +\int_{\Omega}(\varphi(u))^2 \, dx\right),
	\end{split}
\end{equation}
 where  the constant $C$ only depends on $n,s,\Omega,\Lambda$.
 
We pick $\beta_1=\frac{2^*}{2_s^*}>1$ provided that $n>2$ ( $\beta_1$ is any real number in $[2,+\infty)$ if $n=1,2$). Thus, one has that $u\in L^{2_s^*\beta_1}(\Omega).$

	Next, we will find an increasing unbounded sequence $ \beta_m $  such that
\[ u\in L^{2_s^*\beta_{m}}(\Omega), \quad \forall m>1. \]
To this end, let us suppose that~$ \beta>\beta_1 $. Thus, recalling that $\varphi$ depends on $T>0$, as in~\eqref{vsdvdsc}, and using that $ \varphi(u)\leqslant |u|^\beta $ in the right hand side of  \eqref{main estimate1}, letting $ T\rightarrow \infty $ we get
\begin{equation}
	\bigg(\int_{\Omega}|u|^{2^*_s\beta}\,dx\bigg)^{2/2^*_s}\leqslant C\beta \bigg(\int_{\Omega} |u|^{2\beta-1}\,dx +\int_{\Omega} |u|^{2\beta}\, dx\bigg). \label{rough estimate}
\end{equation}
 Using Young's inequality, we also observe that
\begin{equation}\label{estimate of 2beta-1-norm}
		\int_{\Omega}|u|^{2\beta-1}\,dx
		\leqslant \int_{\Omega}|u|^{2\beta}\,dx+|\Omega|.
\end{equation}
Hence, by combining \eqref{rough estimate} with \eqref{estimate of 2beta-1-norm}, we conclude that
\begin{equation}
	\bigg(\int_{\Omega}|u|^{2^*_s\beta}\,dx\bigg)^{2/2^*_s}\leqslant  C\beta \left(|\Omega| +2\int_{\Omega} |u|^{2\beta}\, dx\right).
\end{equation} 
As a consequence,  
exploiting the formula $ (a+b)^2\leqslant 2(a^2+b^2) $, one derives that
\begin{equation}
	\begin{split}
	\left(1+\int_{\Omega} |u|^{2^*_s\beta}\, dx\right)^2&\leqslant 2+2 C^{2_s^*}\beta^{2_s^*} \left(|\Omega| +2\int_{\Omega} |u|^{2\beta}\, dx\right)^{2_s^*}\\
	&\leqslant 2+2\left[2C\beta(|\Omega|+1)\left(1+\int_{\Omega} |u|^{2\beta}\, dx \right)\right]^{2^*_s}.
\end{split}
\end{equation}
Therefore,
\begin{equation}\label{estimate of iteration}
	\bigg(1+\int_{\Omega}|u|^{2^*_s\beta}\,dx\bigg)^{\frac{1}{2^*_s\beta}}\leqslant (C\beta)^{\frac{1}{2\beta}}\bigg(1+\int_{\Omega} |u|^{2\beta}\, dx\bigg)^{\frac{1}{2\beta}}
\end{equation}
where $ C=C(n,s,\Omega,\Lambda)$.

For $ m\geqslant 1, $ we define $ \beta_{m+1} $ such that
\begin{equation*}
	2\beta_{m+1}=2^*_s\beta_{m}.
\end{equation*}
Thus, \eqref{estimate of iteration} becomes
\begin{equation*}
	\bigg(1+\int_{\Omega}|u|^{2^*_s\beta_{m+1}}\,dx\bigg)^{\frac{1}{2^*_s\beta_{m+1}}}\leqslant (C\beta_{m+1})^{\frac{1}{2\beta_{m+1}}}\bigg(1+\int_{\Omega} |u|^{2^*_s\beta_{m}}\, dx\bigg)^{\frac{1}{2^*_s\beta_{m}}}.
\end{equation*}
We now define $ C_{m+1}:=C\beta_{m+1} $ and
\begin{equation}\label{vnsdhvs}
	B_m:=\left(1+\int_{\Omega} |u|^{2^*_s\beta_{m}}\, dx\right)^{\frac{1}{2^*_s\beta_{m}}}.
\end{equation}
In particular, note that $B_1= \bigg(1+\int_{\Omega} |u|^{2^*_s\beta_{1}}\, dx\bigg)^{\frac{1}{2^*_s\beta_{1}}}=\left(1+\|u\|^{2^*}_{L^{2^*}(\Omega)}\right)^{1/2^*}$ is bounded.
Additionally, it is not difficult to verify that there exists a constant $ \tilde{C}>0 $ independent of $m$, such that
\begin{equation}\label{DRI1}
	B_{m+1}\leqslant C_{m+1}^{\frac{1}{2\beta_{m+1}}}B_m\leqslant  \prod_{k=2}^{m+1}C_k^{\frac{1}{2\beta_{k}}}B_1\leqslant \tilde{C} B_1.
\end{equation}
From this and~\eqref{vnsdhvs}, one has that
\begin{equation}\label{vsdvsdfsd}
	\|u\|_{L^{2^*_s\beta_{m+1}}(\Omega)}=\left(\int_{\Omega} |u|^{2^*_s\beta_{m+1}}\, dx\right)^{\frac{1}{2^*_s\beta_{m+1}}}\leqslant \tilde{C}\left(1+\|u\|^{2^*}_{L^{2^*}(\Omega)}\right)^{1/2^*}.
\end{equation}
Thus, by the H\"older inequality, we conclude that $ u\in L^p(\Omega)$ for every~$ p\in[1,+\infty)$. Furthermore, 
a limit argument implies that
\begin{equation}\label{dsgsdf}
	\|u\|_{L^\infty(\Omega)}\leqslant \tilde{C}\left(1+\|u\|^{2^*}_{L^{2^*}(\Omega)}\right)^{1/2^*},
\end{equation}
where $ \tilde{C} $ depends only on $ n,s,\Omega,\Lambda$.

We now claim that 
\begin{equation}\label{vrg1}
	\|u\|_{L^{2^*}(\Omega)}\leqslant {C}(n,q,\Lambda,\Omega)
\end{equation}
for some constant $ {C}(n,q,\Lambda,\Omega)>0$.

By testing equation~\eqref{Maineq} against $u$, we see that 
\begin{equation}
	\begin{split}
			\|u\|^2_{L^{2^*}(\Omega)}&\leqslant C\|\nabla u\|^2_{L^{2}(\Omega)}\leqslant C\Lambda\int_{\Omega} \left(|u|+|u|^{q}\right)\,dx\\
			&\leqslant C\left(\|u\|_{L^{2^*}(\Omega)}+\|u\|^q_{L^{2^*}(\Omega)}\right),
	\end{split}
\end{equation}
where the positive constants $C$ are different and depend  on $n,\Omega,\Lambda,q$.
This implies the desired result~\eqref{vrg1}, thanks to the fact that  $q\in[1,2)$.
From~\eqref{dsgsdf} and~\eqref{vrg1}, we finish the proof of Proposition~\ref{th: regularity for linearity}.
\end{proof}}

\section{$ C^{1,\alpha}$-regularity }\label{sec: C^alpha-regularity}\setcounter{equation}{0}

In this section, we focus on the $ C^{1,\alpha}$-regularity up to the boundary  and on the interior $ C^{1,\alpha}$-regularity for operator $\mathcal{L}$.

\subsection{$ C^{1,\alpha}$-regularity up to the boundary }\label{sec: C^1,alpha-regularity}\setcounter{equation}{0}

The $C^{1,\alpha}$-regularity up to the boundary result for $\mathcal{L}$ is now an immediate corollary of  the following  $W^{2,p} $-regularity theorem.

\begin{Theorem}\label{theorem w 2p}
	Let $ \Omega $ be a $ C^{1,1} $ domain in $ \mathbb{R}^n $ and $ s\in (0,1)$. If $ f\in L^p(\Omega) $ 
	where  $p$ satisfies \begin{equation}
		\begin{cases}\label{eq definition of p}
			1<p<\infty ,\qquad &{\mbox{ if }}s\in(0,\frac{1}{2}] ,     \\
			n<p<\frac{n}{2s-1},& {\mbox{ if }}s\in(\frac{1}{2},1).
		\end{cases}
	\end{equation}
Then,
the problem
	\begin{equation}\label{eq:W{2,p} main equation}
	-\Delta u+(-\Delta)^su=f, \qquad \text{in } \Omega
	\end{equation}
	has a unique solution $ u\in W^{2,p}(\Omega)\cap W^{1,p}_0(\Omega)$.
	
	Furthermore,
	\begin{equation}\label{eq:W 2p estimate}
	\|u\|_{W^{2,p}(\Omega)}\leqslant C\bigg(\|u\|_{L^p(\Omega)}+\|f\|_{L^p(\Omega)}\bigg),
	\end{equation}
	where the positive constant $ C$ depends only on $ \Omega, n, s, p $.
\end{Theorem}

Theorem~\ref{theorem w 2p} can be obtained by  following an argument given in \cite{SVWZ22}.
For the reader's convenience we sketch this argument here, by pointing out the following steps:
\begin{itemize}
	\item[Step 1.] Obtain $ L^p $-estimates on the operator $ (-\Delta)^s$;
	\item[Step~2.] Apply the fixed point theorem  to deduce that the problem
	\begin{equation*}
	-\Delta u+\lambda u=f-(-\Delta)^su, \qquad \text{in } \Omega
	\end{equation*}
	has a unique solution $ u\in W^{2,p}(\Omega)\cap W^{1,p}_0(\Omega) $ for $ \lambda>0$ large enough;
	\item[Step~3.] Employ the maximum principle  in  \cite[Theorem 1.2]{BDVV22b} and the
	bootstrap method to conclude.
\end{itemize}

With this, we can now complete the proof of the $C^{1,\alpha}$-regularity theory by proceeding as follows.

\begin{proof}[Proof of Theorem~\ref{th:C^{1,alpha}}]
Let $ u\in X_0^1 $ be the weak solution of \eqref{Maineq}. Theorem~\ref{th: regularity} gives that the map~$x\mapsto g(x,u(x))$
belongs to~$ L^\infty(\Omega)$.

Moreover, due to Theorem~\ref{theorem w 2p}, we  deduce that  the linear problem
	\begin{equation}
	-\Delta v+(-\Delta)^sv=g(x,u(x)), \qquad \text{in } \Omega
\end{equation}
has a unique solution $ v\in W^{2,p}(\Omega)\cap W^{1,p}_0(\Omega) $ for every $p$ satisfying \eqref{eq definition of p}.
By combining this with the Sobolev inequality, we obtain that
 \begin{itemize}
	\item [(1)] if $s\in (0,\frac{1}{2}]$, then for any $\alpha \in (0,1)$,  $v\in C^{1,\alpha}(\overline\Omega)$;
	
	\item [(2)] if $s\in (\frac{1}{2},1)$, then for any $\alpha\in (0,2-2s)$,  $v\in C^{1,\alpha}(\overline\Omega)$.
\end{itemize}

Finally, the Lax-Milgram Theorem yields that $ u=v $. In particular, according to  \eqref{eq:W 2p estimate}, \eqref{eq: H 1} and the Sobolev inequality, we obtain that
\begin{equation}
	\begin{split}
		\|u\|_{C^{1,\alpha}(\overline{\Omega})}&\leqslant C\left(\|u\|_{L^\infty(\Omega)}+\|g(\cdot,u(\cdot))\|_{L^\infty(\Omega)}\right)
		\leqslant C \left(\|u\|_{L^\infty(\Omega)}+c(1+\|u\|_{L^\infty(\Omega)}^{2^*-1})\right),
	\end{split}
\end{equation}
where the positive constant $C$ depends only on $n,\alpha,\Omega,s$.
\end{proof}

\subsection{Interior  $ C^{1,\alpha} $-regularity}\label{sec: C^1-interior regularity}\setcounter{equation}{0}
The  interior $ C^{1,\alpha} $-regularity for the mixed operator $ \mathcal{L}$ will be derived by a mollifier technique and a cutoff argument. We split the proof of Theorem~\ref{th C^1,alpha interior without boundary condition} into three steps.

\noindent{Step 1.} Regularize the solution $u$ by the standard mollifier to get some properties about $u_\epsilon $. See subsection~\ref{sec: Mollifier argument 1 }.

\noindent{Step 2.} Use the cut off argument for $u_\epsilon$ and employing \cite[Proposition 2.18]{MR4560756} to get the interior $C^{1,\alpha}$-norm  of $u_\epsilon$ in $B_R$ for any $R>0$ small enough. See subsection~\ref{sub A priori C{2,alpha}-estimate 1}.

\noindent{Step 3.}  Apply  Proposition~\ref{pro final estimate 1} and  Arzel\`{a}-Ascoli Theorem  to get the inclusion. See subsection~\ref{sec:Proof of Theorem1.5 1}.

We now introduce some notations.  
\begin{itemize}
\item [(a)] Let $V$ be an open set  with $ V\Subset\Omega $ and denote
\begin{equation}\label{definition of rho 1}
	\rho:=\text{ dist}(V,\partial\Omega),\quad \text{  and } \quad  V_{\delta}:=\left\{x\in \Omega: \text{ dist}(x, V)<\delta\right\}.
\end{equation}
\item [(b)] For any given $x_0\in V_{\frac{\rho}{4}}$ where $V_{\frac{\rho}{4}} $ is given in \eqref{definition of rho 1}, denote
\begin{equation}\label{x_0}
	0<R<\min\left(\frac{1}{2},\frac{\rho}{10}\right),\quad B_{R}(x_0):=\left\{x\in \Omega: |x-x_0|<R \right\}.
\end{equation}
\item [(c)] Denote the interior norm  as follows:
\begin{equation}
	[u]_{\alpha;B_{R}(x_0)}=\sup\limits_{x,y\in B_{R}(x_0)}\frac{|u(x)-u(y)|}{|x-y|^\alpha},\qquad 0<\alpha<1;
\end{equation}
\begin{equation}
	|u|^\prime_{k;B_{R}(x_0)}=\sum\limits_{j=0}\limits^{k}R^j\|D^j u\|_{L^\infty(B_{R}(x_0))};
\end{equation}
\begin{equation}
	|u|^\prime_{k,\alpha;B_{R}(x_0)}=|u|^\prime_{k;B_{R}(x_0)}+R^{k+\alpha}[D^ku]_{\alpha;B_{R}(x_0)}.
\end{equation}
\item[(d)] Let $u\in X^1$ and $\|u\|_{L^\infty(\Omega)}$ be finite.  Denote the interval as follows: \begin{equation}\label{I}
		{\rm I_u}=\left[-\|u\|_{L^\infty(\Omega)},\|u\|_{L^\infty(\Omega)}\right].
	\end{equation}
\end{itemize}
\subsubsection{{A Mollifier technique}}\label{sec: Mollifier argument 1 }		

Let $u\in X^1$ solve the equation \eqref{1.5} and $\eta_\epsilon$ is the standard mollifier. For every $x\in\mathbb{R}^n$, we denote 
\[ u_\epsilon(x):=(\eta_\epsilon\ast u)(x)= \int_{\Omega} \eta_\epsilon(x-y)u(y)\, dy= \int_{|y|\leqslant \epsilon} \eta_\epsilon(y)u(x-y)\, dy\quad \forall 0<\epsilon<R. \]

Recall that if $u$ solves the following equation in the weak sense
\begin{equation}
-\Delta u+(-\Delta)^s u= g(x,u) \qquad \text{  in } \Omega,
\end{equation} 
then, for every $\epsilon\in(0,R)$ small enough, 
\begin{equation}\label{eq u epsilon1}
-\Delta(\eta_\epsilon\ast u)+(-\Delta)^s (\eta_\epsilon \ast u)= \eta_\epsilon \ast g(x,u) \qquad \text{  in } V_{\frac{3\rho}{4}}.
\end{equation}
Let us  define 
$ g_\epsilon:=\eta_\epsilon\ast g(x,u) $. Then, $u_\epsilon$ and $g_\epsilon$ satisfy the following equation 
\begin{equation}
-\Delta u_\epsilon+(-\Delta)^s u_\epsilon= g_\epsilon \qquad \text{  in } V_{\frac{3\rho}{4}}.
\end{equation}
 Moreover, the following regularity estimates on $u_\epsilon$ and $g_\epsilon$ follow as the direct consequences of
the standard properties of the convolutions
\begin{Proposition}\label{pro u epsilon1} Assume that $u$ is a bounded solution of~\eqref{1.5}. 
	\begin{itemize} 
		\item  Then $u_\epsilon \in C^{2}(\overline V_{\frac{3\rho}{4}})\cap L^\infty(\mathbb{R}^n)$. In particular, for every $  \epsilon\in(0,R)$
		\begin{equation}\label{eq:u epsilon C11}
			\|u_\epsilon\|_{L^\infty(\mathbb{R}^n)}\leqslant \|u\|_{L^\infty(\mathbb{R}^n)}.
		\end{equation}
		\item For every $x_0\in V_{\frac{\rho}{4}}$,	let $g\in L^\infty_{\rm loc}(\Omega\times\mathbb{R})$,  for any  $\epsilon\in (0,R)$. 
		Then  \begin{equation}\label{g_epsilon1}
			\|g_\epsilon\|_{L^\infty( B_{R}(x_0))}\leqslant \|g\|_{L^\infty(\overline B_{2R}(x_0) \times {\rm I_u})}.
		\end{equation}
	\end{itemize}
\end{Proposition}

%


\subsubsection{{A priori $C^{1,\alpha}$-estimate}}\label{sub A priori C{2,alpha}-estimate 1}

We now use the cut off argument for $u_\epsilon$ to estimate the $C^{1,\alpha}$-norm  of $u_\epsilon$. 

Consider a cut off function $ \phi\in C^\infty_0(\mathbb{R}^n,\mathbb{R}) $ satisfying 
\begin{equation}\label{cutoff function 21}
\begin{cases}
	\phi \equiv 1 \text{  on } B_{\frac{3}{2}}(0)\\
	\text{supp}(\phi)\subset B_{2}(0);\\
	0\leqslant \phi\leqslant 1 \text{  on } \mathbb{R}^n.
\end{cases}
\end{equation}
 Let us take $x_0\in V_{\frac{\rho}{4}}$ and denote 
\begin{equation}\label{R}
	\phi^R(x):=\phi\left(\frac{x-x_0}{R}\right).
\end{equation}
Thus, one has that
\[ B_{4R}(x_0)\subset V_{\frac{3\rho}{4}}\subset \Omega, \quad \text{supp}(\phi^R)\subset B_{2R}(x_0). \]

\begin{Lemma}\label{lemma v epsilon1} 
We assume that  $ g(\cdot,u(\cdot))\in L^\infty(\Omega) $.
Suppose that  $ u\in C^{2}(\overline V_{\frac{3\rho}{4}})\cap L^\infty(\mathbb{R}^n)$  solves the following equation 
\begin{equation}\label{u is a solution in V1}
	-\Delta u+(-\Delta)^s u= g \qquad \text{  in } V_{\frac{3\rho}{4}}.
\end{equation}

Then,  for every $x_0\in V_{\frac{\rho}{4}}$, there exists~$ \psi\in L^{\infty}( B_{R}(x_0),\mathbb{R}) $ such that $ v:=\phi^R u $ satisfies 
\begin{equation}\label{eq: new equation21}
	-\Delta v+(-\Delta)^s v= \psi \qquad \text{  in } V_{\frac{3\rho}{4}}.
\end{equation}
In addition, 
\begin{equation}\label{estimate psi 21}
	R^2\|\psi\|_{L^\infty(B_{R}(x_0))}\leqslant  C_{\rho,n,s}\left(\| g(\cdot,u(\cdot))\|_{L^\infty(B_{R}(x_0))}+\|u\|_{L^\infty(\mathbb{R}^n)}\right)
\end{equation}
for some positive constant $C_{n,s,\rho}$.

\end{Lemma}

\begin{proof}
From  the definition of the function $ v $, we know that
\begin{equation}\label{eq: function space of v21}
	v\in C^{2}_0( B_{2R}(x_0))\subset C^{2}_0(\mathbb{R}^n). 
\end{equation}
Furthermore, to check~\eqref{eq: new equation21}, one can calculate that  
\begin{equation}\label{eq: new equation:PSI2 11}
	\psi= g(x,u)+\Delta (u(1-\phi^R))-(-\Delta)^s (u(1-\phi^R))\quad  \text{  in } V_{\frac{3\rho}{4}}.
\end{equation} 
We also notice that $\phi^R=1$ in $B_{3R/2}(x_0)$ and that  \begin{equation}\label{eq: new equation:PSI21}
	\psi= g(x,u)-(-\Delta)^s (u(1-\phi^R))\quad  \text{  in } B_{3R/2}(x_0).
\end{equation}

We  now obtain suitable $ L^\infty $- estimates on the function $  (-\Delta)^s (u(1-\phi^R)) $.
Observing that for $x\in B_{R}(x_0)$, we have that
\begin{equation}
	(-\Delta)^s (u(1-\phi^R))(x)=\int_{\mathbb{R}^n\backslash B_{3R/2}(x_0)} \frac{-u(1-\phi^R)(y)}{|x-y|^{n+2s}}\, dy.
\end{equation}
{F}rom this and the fact that \begin{equation}\label{fact}
	|x-y|\geqslant |y-x_0-R|\geqslant |y-x_0|-R\geqslant \frac{1}{5}(R+|y-x_0|)\quad \forall y\in \mathbb{R}^n\backslash B_{3R/2}(x_0), x\in B_{R}(x_0),
\end{equation}
one obtains that
\begin{equation}
	\begin{split}\label{step 2 11}
		\|(-\Delta)^s (u(1-\phi^R))\|_{L^\infty(B_{R}(x_0))}\leqslant& 5^{n+2s}\int_{ \mathbb{R}^n\backslash B_{3R/2}(x_0)}\frac{|u(y)|}{(R+|y-x_0|)^{n+2s}}\, dy\\
		\leqslant& C_{n,s}\|u\|_{L^\infty(\mathbb{R}^n)}R^{-2s}.
	\end{split}
\end{equation}
Since $ g(\cdot,u(\cdot))\in L^\infty(B_{R}(x_0)) $, we have 
\begin{equation}
	R^2\|\psi\|_{L^\infty(B_{R}(x_0))}\leqslant  C_{\rho,n,s}\left(\| g(\cdot,u(\cdot))\|_{L^\infty(B_{R}(x_0))}+\|u\|_{L^\infty(\mathbb{R}^n)}\right)
\end{equation}
for some positive constant $C_{n,s,\rho}$.
\end{proof}
According to \eqref{eq u epsilon1}, we know that 		for every $\epsilon\in(0,R)$ small enough,  $u_\epsilon\in C^{2}(\overline V_{\frac{3\rho}{4}})$ solves the following equation 
\begin{equation}
-\Delta u_\epsilon+(-\Delta)^s u_\epsilon= g_\epsilon \qquad \text{  in } V_{\frac{3\rho}{4}},
\end{equation}
where $g_\epsilon$ satisfies the estimate \eqref{g_epsilon1}.  

Consider the cut off function $\phi^R$ given by \eqref{R}, let $v_\epsilon:=\phi^R u_\epsilon$, then 	applying Lemma~\refeq{lemma v epsilon1}, one has there exists~$ \psi_\epsilon\in L^{\infty}(B_{R}(x_0),\mathbb{R}) $ such that $ v_\epsilon $ satisfies  
\begin{equation}\label{eq: v epsilon equation1}
-\Delta v_\epsilon+(-\Delta)^s v_\epsilon= \psi_\epsilon \qquad \text{  in } V_{\frac{3\rho}{4}},
\end{equation}
where 
\begin{equation}
\psi_\epsilon= g_\epsilon(x,u)+\Delta (u_\epsilon(1-\phi^R))-(-\Delta)^s (u_\epsilon(1-\phi^R))\quad  \text{  in } V_{\frac{3\rho}{4}}.
\end{equation} 
Also, $\phi^R=1$ in $ B_{3R/2}(x_0)$ and 
\begin{equation}\label{eq: new equation:PSI epsilon1}
\psi_\epsilon= 	g_\epsilon(x,u)-(-\Delta)^s (u_\epsilon(1-\phi^R))\quad  \text{  in } B_{3R/2}(x_0).
\end{equation} 
In particular, employing \eqref{estimate psi 21}, one finds that
\begin{equation}
\begin{split}\label{eq:psi epsilon1}
	R^2\|\psi_\epsilon\|_{L^\infty(B_{R}(x_0))}&\leqslant  C_{\rho,n,s} \left(\| g_\epsilon(\cdot,u(\cdot))\|_{L^\infty(B_{R}(x_0))}+\|u_\epsilon\|_{L^\infty(\mathbb{R}^n)}\right),
\end{split}
\end{equation}
where  $ C_{\rho,n,s} $ is a positive constant.

\smallskip
 We notice that,  
for $v_\epsilon\in C^\infty_0( B_{2R}(x_0))$, one has 
\begin{equation}
\begin{split}\label{aa}
	R^2\|(-\Delta)^s v_\epsilon\|_{L^\infty(B_{R}(x_0))}
	&\leqslant C_{n,s,\rho}\left(R^{1+\beta}[Dv_\epsilon]_{0,\beta;B_{2R}(x_0)}+\|v_\epsilon\|_{L^\infty(B_{2R}(x_0))}\right)\\
	&= C_{n,s,\rho}\bigg(\|\phi^R u_\epsilon\|_{L^\infty(B_{2R}(x_0))}+R^{1+\beta}[\phi^R D u_\epsilon+u_\epsilon D\phi^R]_{0,\beta;B_{2R}(x_0)}	\bigg)\\
	&\leqslant C_{n,s,\rho}
	|u_\epsilon|^\prime_{1,\beta;B_{2R}(x_0)},
\end{split}
\end{equation}
for every $\beta\in(\max\left\{0,2s-1\right\},1). $  

As a result, for any given $\alpha\in(\max\left\{0,2s-1\right\},1)$,  taking $\beta=\frac{\max\left\{0,2s-1\right\}+\alpha}{2} $,  invoking \cite[Proposition 2.18]{MR4560756}, and utilizing \eqref{aa} and interpolation inequality, one has   for all~$\delta>0$
there exists $ C_{\delta}>0  $ such that  
\begin{equation}\label{C 2 of v 21}
\begin{split}
	|v_\epsilon|^\prime_{{1,\alpha}; B_{R/2}(x_0)}&\leqslant C_{n,\alpha}\left(\|v_\epsilon\|_{L^\infty(B_{R}(x_0))}+R^2\left(\|\psi_\epsilon\|_{L^\infty(B_{R}(x_0))}+\|(-\Delta)^s v_\epsilon\|_{L^\infty(B_{R}(x_0))}\right)\right)\\
	&\leqslant C_{n,\alpha,\rho,\alpha}\bigg(\|v_\epsilon\|_{L^\infty(B_{R}(x_0))}+R^2\|\psi_\epsilon\|_{L^\infty(B_{R}(x_0))}+|u_\epsilon|^\prime_{1,\beta;B_{2R}(x_0)}\bigg)\\
	&\leqslant C_{n,\alpha,\rho,\alpha}\bigg(\|v_\epsilon\|_{L^\infty(B_{R}(x_0))}+R^2\|\psi_\epsilon\|_{L^\infty(B_{R}(x_0))}\\
	&\qquad\qquad\qquad +\delta |u_\epsilon|^\prime_{1,\alpha;B_{2R}(x_0)} +C_\delta \|u_\epsilon\|_{L^\infty(B_{2R}(x_0))}\bigg). 
\end{split}
\end{equation}
Therefore, for every $x_0\in V_{\frac{\rho}{4}}$, recalling~\eqref{eq:psi epsilon1} and applying Propositions~\ref{pro u epsilon1}, we know that $ \forall \delta>0$, there exists $C_\delta$ such that 
{\begin{equation}\label{final equation1}
	\begin{split}
		&\quad\|u_\epsilon\|^\prime_{{1,\alpha}; B_{R/2}(x_0)}=|v_\epsilon|^\prime_{{1,\alpha}; B_{R/2}(x_0)}\\
		&\leqslant C_{n,s, \alpha,\rho}\left(R^2\|\psi_\epsilon\|_{L^\infty(B_{R}(x_0))}+\delta |u_\epsilon|^\prime_{{1,\alpha}; B_{2R}(x_0)}+C_\delta \|u_\epsilon\|_{L^\infty(B_{2R}(x_0))}\right)\\
		&\leqslant C_{n,s, \alpha,\rho}\bigg(\|g\|_{L^\infty(V_{3\rho/4} \times {\rm I_u})}+C_\delta\|u\|_{L^\infty(\mathbb{R}^n)} +\delta |u_\epsilon|^\prime_{{1,\alpha}; B_{2R}(x_0)}\bigg),
	\end{split}
\end{equation}
for every $R\in\left(0,\frac{\rho}{10}\right)$, $\epsilon\in (0,R)$,  and for some 
constant $ C_{n,s, \alpha,\rho}>0 $.

\subsubsection{Interior $C^{1,\alpha}$-regularity }\label{sec:Proof of Theorem1.5 1}

Based on \eqref{final equation1}, we can obtain the following Proposition~\ref{pro final estimate 1}  by adapting \cite[Lemma 8]{MR3331523}, which can guarantee that the $C^{1,\alpha}$-norm of $u_\epsilon$ is uniformly bounded in $\epsilon$ in some ball.

\begin{Proposition}\label{pro final estimate 1}
	
	Let $y\in \mathbb{R}^n, d>0$, and $u\in C^{1,\alpha}(B_{d}(y))$. Suppose that  for any $\delta > 0 $, there
	exists $\Lambda_\delta> 0$ such that, for any $x\in B_{d}(y)$ and any $r\in(0, d-|x-y|]$, we have
	\begin{equation}\label{eq:assumption u1}
		|u|^\prime_{1,\alpha;B_{r/8}(x)}\leqslant \Lambda_\delta+ \delta |u|^\prime_{1,\alpha;B_{r/2}(x)}.
	\end{equation}
	
	Then, there exist constants $\delta_0, C > 0$, depending only on $n,\alpha,d$, such that
	\begin{equation}
		\|u\|_{C^{1,\alpha}(B_{d/8}(y))}\leqslant C\Lambda_{\delta_0}.
	\end{equation}
\end{Proposition}
\begin{proof}
	We notice that, since $r<d$, one has
	\[ |u|^\prime_{1,\alpha;B_{r/8}(x)}\leqslant  |u|^\prime_{1,\alpha;B_{d}(y)},\]
	which implies that 
	\[ M:=\sup\limits_{x\in B_{d}(y), r\in(0, d-|x-y|]} |u|^\prime_{1,\alpha;B_{r/8}(x)}<+\infty.\]
	
	We now use a covering argument: pick $\lambda\in(0,\frac{1}{2}]$, to be chosen later, and	fixed any $x\in B_d(y)$ and $ r\in(0, d-|x-y|]$, we cover $B_{r/8}(x)$ with finitely many balls $\left\{B_{\lambda r/8}(x_k)\right\}_{k=1}^N$ with $x_k\in B_{r/8}(x) $,  for some $N $ depending only on $\lambda,n$.
	
	Furthermore, we observe that $B_{r/2}(x_k)\subset B_{d}(y)$. Indeed,
	\[ |x_k-y|+\frac{r}{2}\leqslant |x_k-x|+|x-y|+\frac{r}{2}<r+|x-y|<d. \]
	Consequently, since $\lambda\leqslant\frac{1}{2}$, we use  \eqref{eq:assumption u} (with $x = x_k $ and $ r$ scaled to $\lambda r$) to get 
	\begin{equation}\label{eq lambda1}
		|u|^\prime_{1,\alpha;B_{\lambda r/8}(x_k)}\leqslant \Lambda_\delta+ \delta |u|^\prime_{1,\alpha;B_{\lambda r/2}(x_k)}.
	\end{equation}
	Then, using \cite[Lemma 7]{MR3331523} with $\rho:=r/8$ and $\lambda=1/8$, recalling  the definition of $M$, and employing \eqref{eq lambda1}, one has 
	\begin{equation}
		\begin{split}
			|u|^\prime_{1,\alpha;B_{r/8}(x)}&\leqslant C\sum\limits_{k=1}\limits^{N}|u|^\prime_{1,\alpha;B_{\lambda r/8}(x_k)}
			\leqslant C_\lambda N \Lambda_{\delta}+ C_\lambda \delta \sum\limits_{k=1}\limits^{N}|u|^\prime_{1,\alpha;B_{\lambda r/2}(x_k)}\\
			&\leqslant C N \Lambda_{\delta}+ C\delta \sum\limits_{k=1}\limits^{N}|u|^\prime_{1,\alpha;B_{ r/16}(x_k)}
			\leqslant CN \Lambda_{\delta}+ C N\delta M.
		\end{split}
	\end{equation}
	It follows from the definition of $M$ that 
	\[ M\leqslant  C N \Lambda_{\delta}+ C N\delta M,  \]
	by choosing  $\delta=\delta_0:=1/(2CN)$, one has 
	\[ M\leqslant  2C N \Lambda_{\delta_0}.\]
	Thus we have proved that
	\[ |u|^\prime_{1,\alpha;B_{r/8}(x)}\leqslant  2C N \Lambda_{\delta_0}.\qquad \forall x\in B_d(y),  r\in(0, d-|x-y|]. \]
	Setting $x=y, r=d$, we have 
	\[ |u|^\prime_{1,\alpha;B_{d/8}(y)}\leqslant  2 CN \Lambda_{\delta_0}, \]
	which implies the desired result that 
	\[ \|u\|_{C^{1,\alpha}(\overline B_{d/8}(y))}\leqslant  2\max\left\{d^{-1-\alpha},1\right\}C N \Lambda_{\delta_0}.\qedhere\]
\end{proof}

\begin{proof}[Proof of Theorem~\ref{th C^1,alpha interior without boundary condition}]
	First of all, one can choose $d=\rho/5 $ in Proposition~\ref{pro final estimate 1}, so that 
	\begin{equation}\label{d}
		B_{d}(y)\subset V_{\rho/4} \quad \text{and}\quad  d/4<\frac{\rho}{10}, \quad \forall y\in V.
	\end{equation}
	As a consequence, applying  Proposition~\ref{pro final estimate 1}, utilizing \eqref{final equation1},
	we know that  
	\begin{equation}
		\|u_\epsilon\|_{C^{1,\alpha}(\overline B_{\rho/40}(y))}\leqslant C\left(\|g\|_{L^\infty( V_{{3\rho}/{4}} \times{\rm I_u})}+C_{\delta_0} \|u\|_{L^\infty(\mathbb{R}^n)}\right),
	\end{equation}
	for every $y\in V$, where the constant $ C>0 $ depends on $ n,s,\alpha,\rho $.

	Applying the Arzel\`{a}-Ascoli Theorem, we obtain that  $ u\in C^{1,\alpha}(\overline B_{\rho/40}(y)). $ More specifically, one has that
	\begin{equation}\label{final inequality}
		\|u\|_{C^{1,\alpha}(\overline B_{\rho/40}(y))}\leqslant C\left(\|g\|_{L^\infty( V_{{3\rho}/{4}} \times{\rm I_u})}+ \|u\|_{L^\infty(\mathbb{R}^n)}\right),
	\end{equation}
	for every $y\in V$ and $\alpha\in(\max\left\{0,2s-1\right\},1)$, and for	some constant $ C>0 $ depending on $ n,s,\alpha,\rho $. 

Using the cover argument again, there exist  finitely many balls $\left\{B_{\rho/40}(y_k)\right\}_{k=1}^N$ with $y_k\in V $,  for some $N$ depending on $n,\rho,V$, such that 
\[ V\subset \bigcup\limits_{k=1}\limits^NB_{\rho/40}(y_k). \]
Thus, we can obtain
\begin{equation}\label{final inequality 1}
	\|u\|_{C^{1,\alpha}(\overline V)}\leqslant C\left(\|g\|_{L^\infty( V_{{3\rho}/{4}} \times{\rm I_u})}+ \|u\|_{L^\infty(\mathbb{R}^n)}\right),
\end{equation}
for every $\alpha\in(0,1)$,	where the constant $ C>0 $ depends on $ n,s,\alpha,\rho $.
\end{proof}

\section{ $ C^{2,\alpha} $-regularity}\label{sec: C^2- regularity}\setcounter{equation}{0}
In this section, we derive the interior $ C^{2,\alpha} $-regularity and the $ C^{2,\alpha} $-regularity  up to the boundary for the mixed operator $ \mathcal{L}$.

\subsection{Interior  $ C^{2,\alpha} $-regularity}\label{sec: C^2-interior regularity}
The aim of this section is to establish an interior $ C^{2,\alpha} $-regularity theory  for the problem \eqref{Maineq}. 
To this end, we should overcome 
the additional difficulty caused by the nonlinearity $ g(x,u(x))$.
More precisely, since the nonlinearity $g$ depends on unknown function $u$,  both the H\"{o}lder continuity of nonlinearity $ g $ and the regularity of $u$ impact the regularity of $g(x,u(x))$ as a function of x.  For this,  we 	apply a suitable truncation argument again of the solution $ u $ with an interior $ C^{1,\alpha} $-regularity argument to have a version of $C^{2,\alpha}$ regularity result stated in Theorem~\ref{th C^2,alpha interior without boundary condition}.
				
				
We  split the proof of Theorem~\ref{th C^2,alpha interior without boundary condition} into two steps:
				
\begin{itemize}
	\item [(a)]  Employ a mollifier technique and a cut off argument  to get the interior $C^{2,\alpha}$-norm of convolution  $u_\epsilon$ in subsection~\ref{sub A priori C{2,alpha}-estimate}.
	\item [(b)]Apply  Proposition~\ref{pro final estimate} and  Arzel\`{a}-Ascoli Theorem  to get the inclusion in subsection~\ref{sec:Proof of Theorem1.5}.
		\end{itemize}
				
%
%
%

	\subsubsection{{A mollifier technique and a truncation argument}}\label{sub A priori C{2,alpha}-estimate}
			
	Let $u\in X^1$ solve the equation \eqref{1.6}, then the convolutions
	$u_\epsilon$ and $g_\epsilon$ satisfy the following equation 
	\begin{equation}\label{u epsilon}
	-\Delta u_\epsilon+(-\Delta)^s u_\epsilon= g_\epsilon \qquad \text{  in } V_{\frac{3\rho}{4}}.
			\end{equation}

			
 Moreover, from 
		the standard properties of the convolutions, one deduce that
		\begin{Proposition}\label{pro u epsilon2} Assume that $u$ is a bounded solution of~\eqref{1.6}. 
			\begin{itemize} 
				\item  Then $u_\epsilon \in C^{2,\alpha}(\overline V_{\frac{3\rho}{4}})\cap L^\infty(\mathbb{R}^n)$. In particular, for every $  \epsilon\in(0,R)$
				\begin{equation}
					\|u_\epsilon\|_{L^\infty(\mathbb{R}^n)}\leqslant \|u\|_{L^\infty(\mathbb{R}^n)}.
				\end{equation}
	\item  For every  $x_0\in V_{\frac{\rho}{4}}$, let $g\in C^{\alpha}_{\rm loc}(\Omega\times\mathbb{R}^n)$ for any  $\alpha\in (0,1)$. 
	If $u\in C^1_{\rm loc}(\Omega)$,  then for every   $\epsilon\in (0,R)$ 
	\begin{equation}\label{g_epsilon *}
		\|g_\epsilon(\cdot,u(\cdot))\|_{C^{\alpha}(\overline B_{R}(x_0))}\leqslant \|g\|_{C^{\alpha}(\overline B_{2R}(x_0) \times{\rm I_u})}\left(1+\|Du\|_{L^\infty(B_{2R}(x_0))}\right).
	\end{equation}
			\end{itemize}
		\end{Proposition}

						We now use the cut off argument again for $u_\epsilon$ to get the $C^{2,\alpha}$-norm  of $u_\epsilon$. 
						
						
		\begin{Lemma}\label{lemma v epsilon}
	For any  $\alpha\in(0,1)$,	we assume that  $ g(\cdot,u(\cdot))\in C^\alpha_{\rm loc}(\Omega) $.
	Suppose that  $ u\in C^{2,\alpha}(\overline V_{\frac{3\rho}{4}})\cap L^\infty(\mathbb{R}^n)$  solves the following equation 
	\begin{equation}\label{u is a solution in V}
		-\Delta u+(-\Delta)^s u= g \qquad \text{  in } V_{\frac{3\rho}{4}}.
		\end{equation}

	Then,  for every $x_0\in V_{\frac{\rho}{4}}$, there exists~$ \psi\in C^{\alpha}( B_{R}(x_0),\mathbb{R}) $ such that $ v:=\phi^R u $ satisfies 
		\begin{equation}\label{eq: new equation2}
			-\Delta v+(-\Delta)^s v= \psi \qquad \text{  in } V_{\frac{3\rho}{4}},
		\end{equation}
		where $ \phi^R(x) $ is defined in \eqref{R}. In particular, 
		\begin{equation}\label{estimate psi 2}
		R^2|\psi|^\prime_{0,\alpha;B_{R}(x_0)}\leqslant  C_{\rho,n,s}\left(R^2| g(\cdot,u(\cdot))|^\prime_{0,\alpha;B_{R}(x_0)}+\|u\|_{L^\infty(\mathbb{R}^n)}\right)
		\end{equation}
		for some positive constant $C_{n,s,\rho}$.
							
			\end{Lemma}
						
	\begin{proof}
	From  the definition of the function $ v $, we know that
	\begin{equation}\label{eq: function space of v2}
	v\in C^{2,\alpha}_0( B_{2R}(x_0))\subset C^{2,\alpha}_0(\mathbb{R}^n). 
			\end{equation}
							We also notice that $\phi^R=1$ in $B_{3R/2}(x_0)$ and that  \begin{equation}\label{eq: new equation:PSI2}
								\psi= g(x,u)-(-\Delta)^s (u(1-\phi^R))\quad  \text{  in } B_{3R/2}(x_0).
							\end{equation}
							\noindent\textbf{Step 1. $ C^\alpha $- estimates on the function $  g(x,u) $.}
							Since $ g(\cdot,u(\cdot))\in C^\alpha(\overline B_{R}(x_0)) $, we have 
							\begin{equation}
								\begin{split}
									R^2| g(\cdot,u(\cdot))|^\prime_{0,\alpha;B_{R}(x_0)}&=	R^2\| g(\cdot,u(\cdot))\|_{L^\infty(B_{R}(x_0))}+R^{2+\alpha}[g(\cdot,u(\cdot))]_{0,\alpha; B_{R}(x_0)}. 
			\end{split}
			\end{equation}
							
	\noindent\textbf{Step 2. $ C^\alpha $- estimates on the function $  (-\Delta)^s (u(1-\phi^R)) $.}
		Follow the proof of Lemma~\ref{lemma v epsilon1},	one obtains that
		\begin{equation}\label{step 2 1}
			\|(-\Delta)^s (u(1-\phi^R))\|_{L^\infty(B_{R}(x_0))}\leqslant C_{n,s}\|u\|_{L^\infty(\mathbb{R}^n)}R^{-2s}.
		\end{equation}				
		
		For every $x_1,x_2\in B_{R}(x_0)$ with $x_1\neq x_2$, 
							\begin{equation}
								\begin{split}
									&|(-\Delta)^s (u(1-\phi^R))(x_1)-(-\Delta)^s (u(1-\phi^R))(x_2)|\\
									\leqslant & \int_{\mathbb{R}^n\backslash B_{3R/2}(x_0)} |-u(1-\phi^R)(y)| \left|\frac{1}{|x_1-y|^{n+2s}}-\frac{1}{|x_2-y|^{n+2s}}\right|\, dy\\	\leqslant &\int_{|x_1-x_2|>\frac{|x_1-y|}{2}\cap \mathbb{R}^n\backslash B_{3R/2}(x_0)}|u(y)|\left|\frac{1}{|x_1-y|^{n+2s}}-\frac{1}{|x_2-y|^{n+2s}}\right|\, dy\\
	&\quad +\int_{|x_1-x_2|\leqslant\frac{|x_1-y|}{2}\cap \mathbb{R}^n\backslash B_{3R/2}(x_0)}|u(y)|\left|\frac{1}{|x_1-y|^{n+2s}}-\frac{1}{|x_2-y|^{n+2s}}\right|\, dy\\
		:=& I_1+I_2.
		\end{split}
	\end{equation}
	For $I_1$, we see that
	\begin{equation}
	\begin{split}\label{I1}
	I_1 &\leqslant 2\int_{|x_1-x_2|>\frac{|x_1-y|}{2}\cap \mathbb{R}^n\backslash B_{3R/2}(x_0)}\frac{|u(y)||x_1-x_2|^\alpha}{|x_1-y|^{n+2s+\alpha}}\, dy\\
	&\leqslant 2\int_{ \mathbb{R}^n\backslash B_{3R/2}(x_0)}\frac{|u(y)||x_1-x_2|^\alpha}{|x_1-y|^{n+2s+\alpha}}\, dy\\
	&\leqslant  C_{n,s}|x_1-x_2|^\alpha \int_{ \mathbb{R}^n}\frac{|u(y)|}{(R+|y-x_0|)^{n+2s+\alpha}}\, dy\\
	&\leqslant C_{n,s}\|u\|_{L^\infty(\mathbb{R}^n)}|x_1-x_2|^\alpha R^{-2s-\alpha}.
	\end{split}
		\end{equation}
	We also observe that, for every $t\in [0,1]$,
	\[ |tx_1+(1-t)x_2-y|\geqslant |x_1-y|-|x_1-x_2|\geqslant \frac{|x_1-y|}{2}, \quad \text{if} \quad |x_1-x_2|\leqslant\frac{|x_1-y|}{2}, \] so that,
		\begin{equation}
	\begin{split}\label{I2}
	I_2&\leqslant \int_{|x_1-x_2|\leqslant\frac{|x_1-y|}{2}\cap \mathbb{R}^n\backslash B_{3R/2}(x_0)}|u(y)|\left|\int_{0}^{1}\frac{|x_1-x_2|}{|tx_1+(1-t)x_2-y|^{n+2s+1}}\, dt\right|\, dy\\
		&\leqslant C_{n,s}\int_{\mathbb{R}^n\backslash B_{3R/2}(x_0)}|u(y)|\left|\int_{0}^{1}\frac{|x_1-x_2|^\alpha |x_1-y|^{1-\alpha} }{|x_1-y|^{n+2s+1}}\, dt\right|\, dy\\
	&\leqslant C_{n,s}\int_{\mathbb{R}^n\backslash B_{3R/2}(x_0)}|u(y)|\frac{|x_1-x_2|^\alpha  }{|x_1-y|^{n+2s+\alpha}}\, dy\\
	&\leqslant  C_{n,s}|x_1-x_2|^\alpha \int_{ \mathbb{R}^n}\frac{|u(y)|}{(R+|y-x_0|)^{n+2s+\alpha}}\, dy\leqslant C_{n,s}\|u\|_{L^\infty(\mathbb{R}^n)}|x_1-x_2|^\alpha R^{-2s-\alpha}.
		\end{split}
	\end{equation}
	Combining \eqref{I1} and \eqref{I2}, one has 
	\begin{equation}
	[(-\Delta)^s (u(1-\phi^R))]_{0,\alpha;B_{R}(x_0)}\leqslant C_{n,s}\|u\|_{L^\infty(\mathbb{R}^n)}R^{-2s-\alpha}.
		\end{equation}
	As a consequence of this and \eqref{step 2 1}, we have that
	\begin{equation}
			\begin{split}
		&	R^2| (-\Delta)^s (u(1-\phi^R))|^\prime_{0,\alpha;B_{R}(x_0)}\\
		=&	R^2\| (-\Delta)^s (u(1-\phi^R))\|_{L^\infty(B_{R}(x_0))}+R^{2+\alpha}[(-\Delta)^s (u(1-\phi^R))]_{\alpha; B_{R}(x_0)}\\
	\leqslant& C_{n,s}\|u\|_{L^\infty(\mathbb{R}^n)}R^{2-2s}.
	\end{split}
		\end{equation}	
	In light of \textbf{Step 1} and \textbf{Step 2}, we obtain
		\begin{equation}
		\begin{split}
R^2|\psi|^\prime_{0,\alpha;B_{R}(x_0)}&=	R^2\| \psi\|_{L^\infty(B_{R}(x_0))}+R^{2+\alpha}[\psi]_{\alpha; B_{R}(x_0)}\\
	&\leqslant C_{\rho,n,s} \left(R^2| g(\cdot,u(\cdot))|^\prime_{0,\alpha;B_{R}(x_0)}+\|u\|_{L^\infty(\mathbb{R}^n)}\right)
	\end{split}
	\end{equation}
		for some positive constant $C_{n,s,\rho}$.
			\end{proof}

		According to \eqref{u epsilon}, we know that 		for every $0<\epsilon<R$ small enough,  $u_\epsilon\in C^{2,\alpha}(\overline V_{\frac{3\rho}{4}})$ solves the following equation 
		\begin{equation}
		-\Delta u_\epsilon+(-\Delta)^s u_\epsilon= g_\epsilon \qquad \text{  in } V_{\frac{3\rho}{4}},
						\end{equation}
		where $g_\epsilon$ satisfies the estimate \eqref{g_epsilon *}.  
						
		Consider the cut off function $\phi^R$ given by \eqref{cutoff function 21}, let $v_\epsilon:=\phi^R u_\epsilon$, then 	applying Lemma~\refeq{lemma v epsilon}, one has there exists~$ \psi_\epsilon\in C^{\alpha}(B_{R}(x_0),\mathbb{R}) $ such that $ v_\epsilon $ satisfies  
		\begin{equation}\label{eq: v epsilon equation}
		-\Delta v_\epsilon+(-\Delta)^s v_\epsilon= \psi_\epsilon \qquad \text{  in } V_{\frac{3\rho}{4}},
			\end{equation}
							where
		\begin{equation}\label{eq: new equation:PSI epsilon}
		\psi_\epsilon= 	g_\epsilon(x,u)-(-\Delta)^s (u_\epsilon(1-\phi^R))\quad  \text{  in } B_{3R/2}(x_0).
					\end{equation} 
		In particular, employing \eqref{estimate psi 2}, one finds that
		\begin{equation}
		\begin{split}\label{eq:psi epsilon}
		R^2|\psi_\epsilon|^\prime_{0,\alpha;B_{R}(x_0)}&\leqslant  C_{\rho,n,s} \left(R^2| g_\epsilon(\cdot,u(\cdot))|^\prime_{0,\alpha;B_{R}(x_0)}+\|u_\epsilon\|_{L^\infty(\mathbb{R}^n)}\right),
			\end{split}
		\end{equation}
		where  $ C_{\rho,n,s} $ is a positive constant.
							
			\smallskip
		{Observing that 
		for $v_\epsilon\in C^\infty_0( B_{2R}(x_0))$, for all~$\delta>0$,
			there exists $ C_{\delta}>0  $ such that  
			\begin{equation}
			\begin{split}\label{bb}
		R^2|(-\Delta)^s v_\epsilon|^\prime_{0,\alpha;B_{R}(x_0)}&= R^2\|(-\Delta)^s v_\epsilon\|_{L^\infty(B_{R}(x_0))} +R^{2+\alpha}[(-\Delta)^s v_\epsilon]_{\alpha; B_{R}(x_0)}\\
		&\leqslant C_{n,s,\rho}\bigg(R^2\|D^2 v_\epsilon\|_{L^\infty(B_{2R}(x_0))}+\|v_\epsilon\|_{L^\infty(B_{2R}(x_0))}\\
		&\qquad+R^{2+\alpha_0}[D^2 v_\epsilon]_{\alpha_0;B_{2R}(x_0)}+R\|D v_\epsilon\|_{L^\infty(B_{2R}(x_0))}	\bigg)\\
		&= C_{n,s,\rho}\bigg(R^2\|u_\epsilon D^2\phi^R +2Du_\epsilon D\phi^R+\phi^RD^2 u_\epsilon\|_{L^\infty(B_{2R}(x_0))}\\
		&\qquad+\|\phi^Ru_\epsilon\|_{L^\infty(B_{2R}(x_0))}+R\|\phi^R D u_\epsilon+u_\epsilon D\phi^R\|_{L^\infty(B_{2R}(x_0))}
										\\
		&\qquad +R^{2+\alpha_0}[u_\epsilon D^2\phi^R +2Du_\epsilon D\phi^R+\phi^RD^2 u_\epsilon]_{\alpha_0;B_{2R}(x_0)}	\bigg)\\
		&\leqslant C_{n,s,\rho}
		|u_\epsilon|^\prime_{2,\alpha_0;B_{2R}(x_0)}
		\leqslant\delta |u_\epsilon|^\prime_{2,\alpha;B_{2R}(x_0)} +C_\delta \|u_\epsilon\|_{L^\infty(B_{2R}(x_0))},
			\end{split}
			\end{equation}
			where 
		\begin{equation}
			\alpha_0=\begin{cases}
			0\qquad & \alpha<2-2s\\
			\alpha-(1-s) &\alpha \geqslant 2-2s.
				\end{cases}
			\end{equation}
		Thus, combining \cite[Theorem 4.6]{GTbook} 	} and \eqref{bb}, one has for all~$\delta>0$
			there exists $ C_{\delta}>0  $ such that  
			\begin{equation}\label{C 2 of v 2}
			\begin{split}
		|v_\epsilon|^\prime_{2,\alpha;B_{R/2}(x_0)}&\leqslant C_{n,\alpha}\left(\|v_\epsilon\|_{L^\infty(B_{R}(x_0))}+R^2\left(|\psi_\epsilon|^\prime_{0,\alpha;B_{R}(x_0)}+|(-\Delta)^s v_\epsilon|^\prime_{0,\alpha;B_{R}(x_0)}\right)\right)\\
		&\leqslant C_{n,\alpha,\rho,\alpha}\bigg(\|v_\epsilon\|_{L^\infty(B_{R}(x_0))}+R^2|\psi_\epsilon|^\prime_{0,\alpha;B_{R}(x_0)}\\
		&\qquad\qquad\qquad +\delta |u_\epsilon|^\prime_{2,\alpha;B_{2R}(x_0)}+C_\delta \|u_\epsilon\|_{L^\infty(B_{2R}(x_0))}\bigg). 
				\end{split}
							\end{equation}
							
		Therefore, for every $x_0\in V_{\frac{\rho}{4}}$, recalling the definition of $v_\epsilon$ and Proposition~\ref{pro u epsilon2}, owing to~\eqref{estimate psi 2}, we know that $ \forall \delta>0$, there exists $C_\delta$ such that 
		\begin{equation}\label{final equation 2}
	\begin{split}
	&|u_\epsilon|^\prime_{2,\alpha;B_{R/2}(x_0)}=|v_\epsilon|^\prime_{2,\alpha;B_{R/2}(x_0)}\\  \leqslant& \ C_{n,s, \alpha,\rho}\left(R^2|\psi_\epsilon|^\prime_{0,\alpha;B_{R}(x_0)}+\delta |u_\epsilon|^\prime_{2,\alpha;B_{2R}(x_0)}+C_\delta \|u\|_{L^\infty(B_{2R}(x_0))}\right)\\
	\leqslant &\ C\bigg(\|g\|_{C^{\alpha}(\overline B_{2R}(x_0) \times{\rm I_u})}\left(1+\|Du\|_{L^\infty(B_{2R}(x_0))}\right)+\|u\|_{L^\infty(\mathbb{R}^n)}\\
&\qquad +\delta |u_\epsilon|^\prime_{2,\alpha;B_{2R}(x_0)}+C_\delta \|u\|_{L^\infty(B_{2R}(x_0))}\bigg)\\
	\leqslant &\ C\bigg(\|g\|_{C^{\alpha}(\overline V_{\frac{3\rho}{4}} \times{I_u})}\left(1+\|Du\|_{L^\infty( V_{\frac{3\rho}{4}})}\right)+C_\delta \|u\|_{L^\infty(\mathbb{R}^n)}+\delta |u_\epsilon|^\prime_{2,\alpha;B_{2R}(x_0)}\bigg),
					\end{split}
		\end{equation} for every $R\in\left(0,\frac{\rho}{10}\right)$ and $\epsilon\in (0,R)$,
		where the constant $ C>0 $ depends only on $ n,s,\alpha,\rho $.
								}
								
		\subsubsection{Interior $C^{2,\alpha}$-regularity}\label{sec:Proof of Theorem1.5}

		In light of the estimate \eqref{final equation 2}, we can obtain the following Proposition~\ref{pro final estimate}  by adapting Proposition~\ref{pro final estimate 1}, which can guarantee the $C^{2,\alpha}$-norm of $u_\epsilon$ is uniformly bounded in $\epsilon$ in some ball.

			\begin{Proposition}\label{pro final estimate}
									
			Let $y\in \mathbb{R}^n, d>0$, and $u\in C^{2,\alpha}(B_{d}(y))$. Suppose that  for any $\delta > 0 $, there
			exists $\Lambda_\delta> 0$ such that, for any $x\in B_{d}(y)$ and any $r\in(0, d-|x-y|]$, we have
			\begin{equation}\label{eq:assumption u}
			|u|^\prime_{2,\alpha;B_{r/8}(x)}\leqslant \Lambda_\delta+ \delta |u|^\prime_{2,\alpha;B_{r/2}(x)}.
				\end{equation}
									
		Then, there exist constants $\delta_0, C > 0$, depending only on $n,\alpha,d$, such that
			\begin{equation}
			\|u\|_{C^{2,\alpha}(B_{d/8}(y))}\leqslant C\Lambda_{\delta_0}.
					\end{equation}
			\end{Proposition}

\begin{proof}[Proof of Theorem~\ref{th C^2,alpha interior without boundary condition}]
	First of all, one can choose $d=\rho/5 $ in Proposition~\ref{pro final estimate}, so that 
	\begin{equation}\label{d}
	B_{d}(y)\subset V_{\rho/4} \quad \text{and}\quad  d/4<\frac{\rho}{10}, \quad \forall y\in V.
		\end{equation}
	As a consequence, applying  Proposition~\ref{pro final estimate}, and utilizing   \eqref{final equation 2},	we know that  
		\begin{equation}
	\|u_\epsilon\|_{C^{2,\alpha}(\overline B_{\rho/40}(y))}\leqslant C\left(\|g\|_{C^{\alpha}(\overline V_{\frac{3\rho}{4}} \times{I_u})}\left(1+\|Du\|_{L^\infty(V_{\frac{3\rho}{4}})}\right)+C_{\delta_0} \|u\|_{L^\infty(\mathbb{R}^n)}\right)
			\end{equation}
	for every $y\in V$, where the constant $ C>0 $ depends on $ n,s,\alpha,\rho $.

	Applying the Arzel\`{a}-Ascoli Theorem, we obtain that  $ u\in C^{2,\alpha}(\overline B_{\rho/40}(y)). $ In particular,
		one has 
		\begin{equation}\label{final inequality 2}
			\|u\|_{C^{2,\alpha}(\overline B_{\rho/40}(y))}\leqslant C\left(\|g\|_{C^{\alpha}(\overline V_{\frac{3\rho}{4}} \times{I_u})}\left(1+\|Du\|_{L^\infty(V_{\frac{3\rho}{4}})}\right)+ \|u\|_{L^\infty(\mathbb{R}^n)}\right),
				\end{equation}
		for every $y\in V$, and for	some constant $ C>0 $ depending on $ n,s,\alpha,\rho $.

		Using the cover argument again, 
		and employing Theorem~\ref{th C^1,alpha interior without boundary condition}, we can obtain the desired result that 
	\begin{equation}\label{final inequality2}
	\begin{split}
	\|u\|_{C^{2,\alpha}(\overline V)}&\leqslant C\left(\|g\|_{C^{\alpha}(\overline V_{\frac{3\rho}{4}} \times{I_u})}\left(1+\|Du\|_{L^\infty(V_{\frac{3\rho}{4}})}\right)+ \|u\|_{L^\infty(\mathbb{R}^n)}\right)\\
	&\leqslant C\left(\|g\|_{C^{\alpha}(\overline V_{\frac{3\rho}{4}} \times{I_u})}\left(1+	\|g\|_{L^{\infty}(\overline V_{\frac{7\rho}{8}} \times{I_u})}+ \|u\|_{L^\infty(\mathbb{R}^n)}\right)+ \|u\|_{L^\infty(\mathbb{R}^n)}\right)\\
	&\leqslant C\left(\|u\|_{L^\infty(\mathbb{R}^n)}+\|g\|_{C^{\alpha}(\overline V_{\frac{7\rho}{8}} \times{I_u})}\right) \left(1+\|g\|_{C^{\alpha}(\overline V_{\frac{7\rho}{8}} \times{I_u})}\right)
				\end{split}
			\end{equation}
	where the constant $ C>0 $ depends on $ n,s,\alpha,\rho $.
		\end{proof}

As a straightforward corollary of Theorems~\ref{th:C^{1,alpha}}	and~\ref{th C^2,alpha interior without boundary condition}, we obtain the following conclusion containing the local and global regularity results of weak solution, which plays a pivotal role in establishing the strong maximum principle and the principal eigenvalue problem stated in the Appendix.
	\begin{Theorem}\label{th C^2,alpha interior}
		Let $ u$  be   a weak solution of \eqref{Maineq}. Assume that  $ \partial\Omega $ is  of class  $ C^{1,1} $ and 
		that {$ g\in C^\alpha_{\rm loc}(\Omega\times\mathbb{R}) $ } satisfies \eqref{eq: H 1} for any given $\alpha\in(0,1)$. 						Then $u\in C_{\rm loc}^{2,\alpha}(\Omega)\cap C(\mathbb{R}^n)$.
\end{Theorem}

\subsection{$ C^{2,\alpha} $-regularity up to the boundary}\label{sec: C^2alpha-global regularity}
The main ideas used here come from {\rm \cite[Theorem B.1]{FABEJK102}.
 We split the proof of Theorem~\ref{theorem C^2 alpha global} into the following two subsections:
\begin{itemize}
	\item[\ref{subsubsection 4.21}:]Estimate $ \|(-\Delta)^{s} u\|_{C^\alpha(\overline{\Omega})} $.
	\item[\ref{subsubsection 4.22}:] Apply the method of continuity to complete the proof of Theorem~\ref{theorem C^2 alpha global}.
\end{itemize}

\subsubsection{H\"{o}lder estimate of the fractional Laplacian term}\label{subsubsection 4.21}

Let $ \alpha\in(0,1) $, we consider the function space ${\mathcal{B}}_\alpha(\Omega)  $ defined as follows:
\begin{equation}\label{definition of B}
 {\mathcal{B}}_\alpha(\Omega):=\left\{u\in C(\mathbb{R}^n):u\equiv 0 \text{ in } \mathbb{R}^n\backslash \Omega \text{  and } u|_{\Omega}\in C^{2,\alpha}(\overline\Omega) \right\}.
\end{equation}
\begin{Lemma}\label{lemma C^alpha estimate }
	Let $ \Omega $ be a $ C^{2,\alpha} $ domain in $ \mathbb{R}^n $, $ u\in {\mathcal{B}}_\alpha(\Omega)$,  $ s\in(0,\frac{1}{2}) $ and $ \alpha+2s<1 $.
	
	Then, $ (-\Delta)^{s} u\in C^\alpha(\overline{\Omega})  $ and
		\begin{equation}\label{eq: C^alpha norm estimate of (0,1/2)}
		\|(-\Delta)^{s} u\|_{C^\alpha(\overline{\Omega})}\leqslant C\|u\|_{C^{1}(\overline{\Omega})},
		\end{equation}	
		where $ C $ depends only on $ n,s,\alpha,\Omega$.
\end{Lemma}

We split the proof of Lemma~\ref{lemma C^alpha estimate } into two steps:
\begin{itemize}
	\item[{Step 1.}] Estimate $ [(-\Delta)^{s} u]_{C^\alpha(\overline{\Omega})} $ for every $ s\in(0,1) $ (see Lemma~\ref{lemma: C^alpha estimate of fraction laplacian }).
	\item[Step~2.] Estimate the $ L^\infty $-norm of $ (-\Delta)^{s} u $ for every $ s\in(0,1) $ (see Lemma~\ref{lemma: boundedness}).
\end{itemize}

\begin{Lemma}\label{lemma: C^alpha estimate of fraction laplacian }
	Suppose that $ \Omega $ is a $ C^{2,\alpha} $ domain in $ \mathbb{R}^n$, $ u\in {\mathcal{B}}_\alpha(\Omega) $, $ s\in(0,\frac{1}{2}) $ and $ \alpha+2s\leqslant 1 $.
	
	Then,
		\begin{equation}\label{eq: C^alpha estimate of (0,1/2)}
		[(-\Delta)^{s} u]_{C^\alpha(\mathbb{R}^n)}\leqslant C[u]_{C^{1}(\overline{\Omega})},
		\end{equation}	
		where $ C $ depends only on $ n,s,\alpha ,\Omega$.
\end{Lemma}
\begin{proof}
	Let us estimate the difference $ (-\Delta)^{s}u(x_1)-(-\Delta)^{s}u(x_2) $ for $ x_1,x_2\in\mathbb{R}^n  $. For this,  we denote  $B_r:=\left\{x\in\mathbb{R}^n:\,|x|<r\right\}$ for some $r>0$.  Note that
	\begin{equation*}
	|(-\Delta)^{s}u(x_1)-(-\Delta)^{s}u(x_2)|\leqslant A_1+A_2,
	\end{equation*}
	where 	\begin{align*}
	A_1&= \left|\int_{B_r}\frac{u(x_1)-u(x_1+z)-u(x_2)+u(x_2+z)}{|z|^{n+2s}}\,dz\right|,\\
	A_2&= \left|\int_{\mathbb{R}^n\backslash B_r}\frac{u(x_1)-u(x_1+z)-u(x_2)+u(x_2+z)}{|z|^{n+2s}}\,dz\right|.
	\end{align*}
	\textbf{Case 1.} Let $ \beta:=\alpha+2s\leqslant 1 $.
	For $ A_1 $,  we use that $ \left|u(x_i)-u(x_i+z)\right|\leqslant[u]_{C^\beta(\overline{\Omega})}|z|^\beta $ for $ i=1,2 $. Therefore,
	\begin{equation}
	\begin{split}\label{eq: A_1 (0,1/2)}
	A_1&\leqslant \int_{B_r}\frac{2[u]_{C^\beta(\overline{\Omega})}|z|^\beta}{|z|^{n+2s}}\,dz
	\leqslant C(n,s,\alpha)[u]_{C^\beta(\overline{\Omega})}r^\alpha.
	\end{split}
	\end{equation}
	For $ A_2 $, we observe that 
\[ \left|u(x_1)-u(x_2)\right|\leqslant[u]_{C^\beta(\overline{\Omega})}|x_1-x_2|^\beta, \] 
	and that  
	\[ \left|u(x_1+z)-u(x_2+z)\right|\leqslant[u]_{C^\beta(\overline{\Omega})}|x_1-x_2|^\beta. \]
 As a consequence, 
	\begin{equation}
	\begin{split}\label{eq: A_2 (0,1/2)}
	A_2&\leqslant \int_{\mathbb{R}^n\backslash B_r}\frac{2[u]_{C^\beta(\overline{\Omega})}|x_1-x_2|^\beta}{|z|^{n+2s}}\,dz
	\leqslant C(n,s)[u]_{C^\beta(\overline{\Omega})}r^{-2s}|x_1-x_2|^\beta.
	\end{split}
	\end{equation}
	Taking $ r=|x_1-x_2| $ and adding $ A_1 $ and $ A_2 $, we obtain
	\begin{equation}\label{eq: C^alpha<C^1 (0,1/2)}
	[(-\Delta)^{s} u]_{C^\alpha(\mathbb{R}^n)}\leqslant C_{n,s,\alpha}[u]_{C^\beta(\overline{\Omega})}\leqslant C_{n,s,\alpha,\Omega}[u]_{C^1(\overline{\Omega})},
	\end{equation}
	where $ C_{n,s,\alpha,\Omega}>0 $ is a suitable constant.
\end{proof}

\begin{Lemma}\label{lemma: boundedness}
		Let $ \Omega\subset\mathbb{R}^n  $ be a $ C^{2,\alpha} $ domain and $ u\in {\mathcal{B}}_\alpha(\Omega). $    Then, for every $ s\in(0,\frac{1}{2}) $ we have that $ (-\Delta)^{s} u\in L^\infty(\mathbb{R}^n) $ and
	\begin{equation}\label{eq:boundedness}
	\|(-\Delta)^{s} u\|_{L^\infty(\mathbb{R}^n)}\leqslant C\|u\|_{C^1(\overline{\Omega})},
	\end{equation}
	where $ C>0 $ depends only on $ n,s,\alpha,\Omega$.
\end{Lemma}
\begin{proof}
	 We choose $ x_0\in\Omega $ such that
	 \begin{equation*}
	 	d_0:=\text{dist}(x_0,\partial\Omega)=\sup\limits_{x\in{\Omega}}\text{dist}(x,\partial\Omega).
	 \end{equation*}
Then, we have that
	\begin{equation}
		\begin{split}\label{eq:(-Delta)^{s}u(x_0)}
		|(-\Delta)^{s}u(x_0)|&\leqslant\left|\int_{|z|<d_0}\frac{u(x_0+z)-u(x_0)}{|z|^{n+2s}}\,dx\right|+\left|\int_{|z|\geqslant d_0}\frac{u(x_0+z)-u(x_0)}{|z|^{n+2s}}\,dx\right|\\
		&\leqslant \int_{|z|<d_0}\frac{\int_{0}^{1}|\nabla u(x_0+tz)|\,dt}{|z|^{n+2s-1}}\, dx+ 2\|u\|_{L^\infty(\Omega)}\int_{|z|\geqslant d_0}\frac{1}{|z|^{n+2s}}\, dx\\
	 &\leqslant c_{n,s}(d_0^{1-2s}+d_0^{-2s})\|u\|_{C^1(\overline\Omega)},
		\end{split}
	\end{equation}
	where $ c_{n,s}>0 $ is another  constant independent of $ u $.
	
	Combining Lemma~\ref{lemma: C^alpha estimate of fraction laplacian } with \eqref{eq:(-Delta)^{s}u(x_0)},
	for every $ x\in\overline{\Omega} $, one has
	\begin{equation*}
		\begin{split}
		|(-\Delta)^{s}u(x)|&\leqslant[(-\Delta)^{s} u]_{C^\alpha(\mathbb{R}^n)}|x-x_0|^\alpha+|(-\Delta)^{s}u(x_0)|\\
		&\leqslant C_{n,s,\alpha}\text{  diam}(\Omega)^\alpha\|u\|_{C^{1}(\overline{\Omega})}+|(-\Delta)^{s}u(x_0)|\\
		&\leqslant C_{n,s,\alpha,\Omega}\|u\|_{C^{1}(\overline{\Omega})}+c_{n,s}(d_0^{1-2s}+d_0^{-2s})\|u\|_{C^1(\overline\Omega)}.
		\end{split}
	\end{equation*}
This implies the desired result~\eqref{eq:boundedness}.
\end{proof}
	\begin{proof}[Proof of Lemma~\ref{lemma C^alpha estimate }]
In the light of  Lemmata~\ref{lemma: C^alpha estimate of fraction laplacian } and~\ref{lemma: boundedness}, we finally obtain that
	\begin{equation*}
		\begin{split}
	\|(-\Delta)^{s}u\|_{C^{\alpha}(\overline{\Omega})}&=\|(-\Delta)^{s}u\|_{L^\infty(\overline{\Omega})}+[(-\Delta)^{s} u]_{C^\alpha(\overline{\Omega})}
	\leqslant C_{n,s,\Omega,\alpha}\|u\|_{C^{1}(\overline{\Omega})},
		\end{split}
	\end{equation*}
for some suitable $ C_{n,s,\Omega,\alpha}>0 $.
\end{proof}

	\subsubsection{Proof of the $ C^{2,\alpha} $-regularity up to the boundary}\label{subsubsection 4.22}
	
	First of all, we establish the following facts:
	\begin{Lemma}\label{lemma facts relating to method of continutiy}
		 For every $ t\in[0,1] $, we set 
		\begin{equation*}
			\mathcal{L}_t=(1-t)(-\Delta)+t\mathcal{L}=-\Delta+t(-\Delta)^s.
		\end{equation*}
	Then, we have that
		\begin{itemize}
			\item[(1)] $ \mathcal{L}_t u\in C^\alpha(\overline{\Omega}) $ for every $ u\in {\mathcal{B}}_\alpha(\Omega) $;
			\item[(2)] there exists a constant $ C=C(n,s,\alpha,\Omega)>0 $ such that
			\begin{equation}\label{eq:C^2 estimate of u }
				\|u\|_{C^{2,\alpha}(\overline{\Omega})}\leqslant C\left(\|\mathcal{L}_tu\|_{C^{\alpha}(\overline{\Omega})}+\sup\limits_{{\Omega}}|u|\right)\qquad \forall u\in {\mathcal{B}}_\alpha(\Omega).
			\end{equation}
		\end{itemize}
			\end{Lemma}
		\begin{proof}
	As regards (1), it follows directly from Lemma~\ref{lemma C^alpha estimate } and the fact that $ \Delta u\in C^\alpha(\overline{\Omega}) $  if $ u\in {\mathcal{B}}_\alpha(\Omega) $.  We now turn to prove \eqref{eq:C^2 estimate of u }.
	
	To this end, we observe that $ \partial\Omega $ is of class $ C^{2,\alpha} $, which allows us to apply \cite[Theorem 6.14]{GTbook}: there exists a constant $ C=C(n,\alpha)>0 $ such that
	\begin{equation}\label{eq:Laplacian estimate of C^2}
		\|u\|_{C^{2,\alpha}(\overline{\Omega})}\leqslant C\left(\|\Delta u\|_{C^{\alpha}(\overline{\Omega})}+\sup\limits_{{\Omega}}|u|\right),
	\end{equation}
	for every $ u\in C^{2,\alpha}(\overline{\Omega}) $ satisfying $ u\equiv 0 $ in $ \mathbb{R}^n\backslash\Omega. $	 
	
	Thus, combining Lemma~\ref{lemma C^alpha estimate } and \eqref{eq:Laplacian estimate of C^2},  for every $ u\in {\mathcal{B}}_\alpha(\Omega) $ one has
\begin{equation}
		\begin{split}\label{eq:Laplacian estimate of C^2 alpha}
		\|u\|_{C^{2,\alpha}(\overline{\Omega})}&\leqslant C\left(\|\Delta u\|_{C^{\alpha}(\overline{\Omega})}+\sup\limits_{{\Omega}}|u|\right)
		=C\left(\|\mathcal{L}_tu-t(-\Delta)^s u\|_{C^{\alpha}(\overline{\Omega})}+\sup\limits_{{\Omega}}|u|\right)\\
		&\leqslant C\left(\|\mathcal{L}_tu\|_{C^{\alpha}(\overline{\Omega})}+\|(-\Delta)^s u\|_{C^{\alpha}(\overline{\Omega})}+\sup\limits_{{\Omega}}|u|\right)\\
		&\leqslant \bar{C}\left(\|\mathcal{L}_tu\|_{C^{\alpha}(\overline{\Omega})}+\| u\|_{C^{1}(\overline{\Omega})}+\sup\limits_{{\Omega}}|u|\right).
		\end{split}
	\end{equation}
		Owing again to the regularity of $ \partial\Omega $, we can utilize the global interpolation inequality (see
		e.g.~\cite[Chapter 6]{GTbook}): there exists a constant $ \theta=\theta(\bar{C})>0 $ such that
		\begin{equation*}
		\| u\|_{C^{1}(\overline{\Omega})}\leqslant \frac{1}{2\bar{C}}\| u\|_{C^{2,\alpha}(\overline{\Omega})}+\theta \sup\limits_{{\Omega}}|u|.
		\end{equation*}
		As a consequence of this and \eqref{eq:Laplacian estimate of C^2 alpha}, we obtain the estimate in~\eqref{eq:C^2 estimate of u }, as desired.
		\end{proof}

			We now use the method of continuity (See e.g. \cite[Theorem 5.2]{GTbook}) to complete the proof.
			
			We first notice that $ {\mathcal{B}}_\alpha(\Omega) $ is  endowed with a structure of Banach space by the norm
			\begin{equation*}
				\| u\|_{{\mathcal{B}}_\alpha(\Omega)}:=\| u\|_{C^{2,\alpha}(\overline{\Omega})}\qquad \forall u\in {\mathcal{B}}_\alpha(\Omega).
			\end{equation*}
				Furthermore, using~\cite[Theorem 4.7]{BDVV22b}, it is not difficult to see that
			\begin{equation}\label{eq: sup u in Omega}
				\sup\limits_{{\Omega}}|u|=\sup\limits_{\mathbb{R}^n}|u|\leqslant\rho \|\mathcal{L}_t u\|_{L^\infty(\Omega)}\leqslant\rho\|\mathcal{L}_t u\|_{C^{\alpha}(\overline{\Omega})},
			\end{equation}
			where $ \rho $ is an another constant independent of $ u $ and $ t $.

 In view of Lemma~\ref{lemma facts relating to method of continutiy} and \eqref{eq: sup u in Omega}, we can apply the method of continuity in the following settings:
\begin{itemize}
	\item[(i)] $ {\mathcal{B}}_\alpha(\Omega) $ and $ C^\alpha(\Omega) $ are Banach space;
	\item[(ii)] $ \mathcal{L}_0 $ and $ \mathcal{L}_1 $ are linear and bounded from $ {\mathcal{B}}_\alpha(\Omega) $ to $ C^\alpha(\overline\Omega) $;
	\item[(iii)] There exists a constant $ \hat{C}>0 $ such that
	\begin{equation*}
		 \| u\|_{{\mathcal{B}}_\alpha(\Omega)}\leqslant \hat{C}\|\mathcal{L}_t u\|_{C^{\alpha}(\overline{\Omega})}\qquad \forall u\in {\mathcal{B}}_\alpha(\Omega) \text{ and } t\in[0,1].
	\end{equation*}
\end{itemize}
 Since $  \mathcal{L}_0=-\Delta $ is surjective, one can deduce that $ \mathcal{L}_1=\mathcal{L} $ is  also surjective: for every $ f\in C^\alpha(\overline{\Omega}) $ there exists a unique $ u\in {\mathcal{B}}_\alpha(\Omega) $ such that
 \begin{equation}
 	\mathcal{L}u=f \qquad \text{ a.e. in  } \Omega.
 \end{equation}
	We  observe that $ u\in C^{2,\alpha}(\overline{\Omega}) $ and $ u\equiv0$ in $ \mathbb{R}^n\backslash \Omega $.
	
\begin{proof}[Proof of Theorem~\ref{theorem C^2 alpha global}]	
	Let $ u\in X_0^1 $ be the weak solution of \eqref{Maineq}. Theorem~\ref{th:C^{1,alpha}} implies that the map~$ x\mapsto g(x,u(x))$ lies in~$ C^\alpha(\overline{\Omega})$. 
	
	Moreover, there exists unique solution  $ v\in {\mathcal{B}}_\alpha(\Omega)\subset X_0^1.  $ In addition, the Lax-Milgram Theorem yields that $u=v$, and in particular,  by combining Theorem~\ref{th:C^{1,alpha}} with
	\eqref{eq:C^2 estimate of u }, we deduce that 
\begin{equation}
		\begin{split}
				\|u\|_{C^{2,\alpha}(\overline{\Omega})}&\leqslant C\left(\|\mathcal{L}u\|_{C^{\alpha}(\overline{\Omega})}+\sup\limits_{{\Omega}}|u|\right)
				\leqslant C\left(\|g(\cdot,u(\cdot))\|_{C^{\alpha}(\overline{\Omega})}+\sup\limits_{{\Omega}}|u|\right)\\
				&\leqslant C \left([g]_{C^\alpha(\overline{\Omega}\times{\rm I_u})}\left(1+\|Du\|_{L^\infty(\Omega)}\right)+\|g\|_{L^\infty(\overline\Omega\times {\rm I_u})}+\|u\|_{L^\infty(\Omega)}\right)\\
				&\leqslant C \left(\|g\|_{C^\alpha(\overline{\Omega}\times{\rm I_u})}\left(1+\|u\|_{C^1(\overline\Omega)}\right)+\|u\|_{L^\infty(\Omega)}\right)\\
				&\leqslant C \left(\|g\|_{C^\alpha(\overline{\Omega}\times{\rm I_u})}\|u\|^{2^*}_{L^\infty(\Omega)}+\|g\|_{C^\alpha(\overline{\Omega}\times{\rm I_u})}+\|u\|_{L^\infty(\Omega)}\right)
			\end{split}
	\end{equation}
	for some constant $C$ depending only on $n,s,c,\alpha,\Omega$.
	Hence, the proof of Theorem~\ref{theorem C^2 alpha global} is completed.
\end{proof}
	
	{
	\begin{appendix}
	
	\section{The strong maximum principle}\setcounter{equation}{0}
	
	We collect here some auxiliary results which follow as a simple byproduct of our main theorems.
	These results are not stated in their full generality, but rather in a way which makes them directly utilizable in nonlinear analysis problems (in particular, we will use them in the present form in the article~\cite{SVWZ}).

First, we present a strong maximum principle (see~\cite{BDVV22b} for a weak maximum principle).

	\begin{Theorem}[Strong maximum principle]\label{coro: solution u>0}
	Let 
	$\Omega$ be a bounded $C^{1,1}$ domain, 
	$ u$  be   a weak solution of \eqref{Maineq}.
	Assume that
	 $ g\in C^\alpha_{\rm loc}(\Omega\times\mathbb{R}) $  satisfies \eqref{eq: H 1} for any given  $\alpha\in(0,1)$
		and   $ g\geqslant 0 $.
		
		Then, $u\in C^{2, \alpha}_{\rm loc}(\Omega)\cap C(\mathbb{R}^n)  $ for any $\alpha\in(0,1)$. Furthermore,  we have that if 
		$u\not\equiv 0 $ in $ \mathbb{R}^n, $
		 then
		$ u>0$ in $ \Omega. $
	\end{Theorem}	 
		\begin{proof}
		
		Since $ u $ is a nontrivial weak solution of \eqref{Maineq}, utilizing Theorems~\ref{th: regularity} and \ref{th C^2,alpha interior}, one deduces that
		\begin{equation*}
			u\in C^{2,\alpha}_{\rm loc}(\Omega)\cap C(\mathbb{R}^n).
		\end{equation*}
		Moreover, due to $ g(x,u)\geqslant 0 $,
		\begin{equation*}
			-\Delta u+(-\Delta)^su\geqslant 0, \text{ a.e. in } \Omega \text{ and }  u\equiv 0  \text{ in } \mathbb{R}^n\backslash \Omega.
		\end{equation*}
		 If 
			$u\not\equiv 0 $ in $ \mathbb{R}^n $, suppose by contradiction that there exists a point $ \eta\in\Omega $ such that  $ u(\eta)\leqslant0 $. Since $ u\in C(\mathbb{R}^n) $ and $ \overline{\Omega} $ is compact,  one can find  $ x_0\in\overline\Omega $ such that
		$$ u(x_0)=\min\limits_{\overline\Omega} u(x)\leqslant u(\eta)\leqslant 0. $$
		We notice that $x_0\in \Omega$ because $u$ is continuous and $ u=0 $ in $ \mathbb{R}\backslash\Omega$.  
		
		Thus, $ \Delta u(x_0)\geqslant 0 $ and so
		\begin{equation}\label{eq: contradiction for u }
			\begin{split}
				0\leqslant \big(-\Delta +(-\Delta)^s \big) u(x_0) &=-\Delta u(x_0)+ \int_{\mathbb{R}^{n}}\frac{u(x_0)-u(y)}{|x_0-y|^{n+2s}}\, dy\\
				&\leqslant \int_{\mathbb{R}^{n}}\frac{u(x_0)-u(y)}{|x_0-y|^{n+2s}}\, dy \leqslant 0.
			\end{split}
		\end{equation}Here, the last inequality holds due to the fact that $ u(x_0)\leqslant u(x)$ for all $ x\in\mathbb{R}^n$.
		Hence, we obtain
		that $ u\equiv 0 $ throughout $ \mathbb{R}^n $, which
		is a contradiction.
%
\end{proof}

\section{The principal eigenvalue problem}

Now we give a characterization
of the principal eigenvalue and eigenfunction. 

\begin{Theorem}[Variational properties for the principal eigenvalue]\label{th: eigenvalue problem}		
Consider the eigenvalue problem 	 	\begin{equation}\label{eq:eigenvalue problem}
	 	\begin{cases}
	 		-\Delta u+(-\Delta)^s u=\lambda u\quad &\text{in } \Omega,\\
	 		u=0&\text{in }\mathbb{R}^n\backslash \Omega.
	 	\end{cases}
	 \end{equation}

Then,	
		\begin{itemize}
	\item [(a)] We have 
		\begin{equation*}\label{definition of eigenvalue}
	\lambda_1:=\min\left\{\|u\|^2_{X_0^1}:\,  u\in X^1_0,\, \|u\|_{L^2(\Omega)}=1 \right\}>0.
			\end{equation*}
		\item [(b)] $ \lambda_1 $ is simple.
			\item[(c)]  Any first eigenfunction $\varphi_1$ is of $C^{2,\alpha}_{\rm loc}(\Omega)$ class  for any $\alpha\in(0,1)$.
				\item[(d)]   Either $\varphi_1\geqslant 0$ or $\varphi_1\leqslant 0$.
		\end{itemize}	
		Furthermore, if $\Omega$  is a $C^{1,1}$ domain, then
	\begin{itemize}
		\item[(e)]  $\varphi_1\in C^{2,\alpha}_{\rm loc}(\Omega)\cap C^{1,\alpha}(\overline \Omega)$  for any $\alpha\in(0,1)$.
			\item[(f)]  $\varphi_1=0$ in $\mathbb{R}^n\backslash\Omega$, and either $\varphi_1>0$ in $\Omega$ or $\varphi_1<0$ in $\Omega$.
				\end{itemize}  
	\end{Theorem}

	\begin{proof}
		As for \textbf{(a)} and \textbf{{(d)}}, we set
		\begin{equation*}
			I(u):=\|u\|^2_{X_0^1}, \qquad  \mathcal{S}= \left\{u\in X_0^1, \|u\|_{L^2(\Omega)}=1  \right\}.
		\end{equation*}
		Let  $ \left\{u_j\right\}\subset \mathcal{S} $ be a minimizing sequence for $ I $, that is 
		\begin{equation*}
			\lim\limits_{j \rightarrow \infty}I{(u_j)}=\mathop{\text{inf }}\limits_{\mathcal{S} } \|u\|^2_{X_0^1}\geqslant 0.
		\end{equation*}
		Then $ \left\{u_j\right\} $ is bounded in $ X_0^1 $. Up to a subsequence, still defined by $ \left\{u_j\right\} $, there exists $ u^*\in X_0^1 $ such that 
		\begin{equation}
			\begin{split}
				u_j\rightharpoonup u^* &\text{ in }   X_0^1;\\
				u_j\rightarrow u^*& \text{ in } L^2(\Omega).
			\end{split}
		\end{equation}
		Since $ \|u_j\|_{L^2(\Omega)}=1 $, we know that  $ \|u^*\|_{L^2(\Omega)}=1  $ and  $ u^*\in \mathcal{S} $.  According to the weak lower semi-continuity of the norm, one has that
		\begin{equation}
			\lim\limits_{j \rightarrow \infty}I{(u_j)}\geqslant I(u^*)\geqslant \mathop{\text{inf }}\limits_{\mathcal{S} } \|u\|^2_{X_0^1},
		\end{equation}
	which  implies  that $ \lambda_1= \mathop{\text{inf }}\limits_{\mathcal{S} } \|u\|^2_{X_0^1}=I(u^*).$ 
		As a consequence of this,  the function $u^*$ can be chosen as
		 the first eigenfunction.

		In order to  prove that~$ \varphi_1\geqslant0 $ a.e. in $ \Omega $, we first need to show that if $ \varphi_1 $ is an eigenfunction related to $ \lambda_1 $, 
		then both $ \varphi_1 $ and $ |\varphi_1| $ realize the minimum in the definition of $\lambda_1$.
		
		Indeed, let  $ x_1\in\left\{x\in\mathbb{R}^n,\varphi_1 >0\right\} $ and $ x_2\in\left\{x\in\mathbb{R}^n,\varphi_1 <0\right\} $. Then,
		\begin{equation}\label{eq:eigen}
			\left||\varphi_1(x_1)|-|\varphi_1(x_2)|\right|=|\varphi_1(x_1)+\varphi_1(x_2)|<\varphi_1(x_1)-\varphi_1(x_2)=|\varphi_1(x_1)-\varphi_1(x_2)|.
		\end{equation}
		If both the sets $ \left\{x\in\mathbb{R}^n,\varphi_1 >0\right\} $ and $ \left\{x\in\mathbb{R}^n,\varphi_1 <0\right\} $  have positive measure, exploiting  \eqref{eq:eigen},
		we get $ I(|\varphi_1|)<I(\varphi_1)=\lambda_1 $, whence, since $ |\varphi_1|\in\mathcal{S} $, we have that~$ I(|\varphi_1|)=I(\varphi_1)=\lambda_1 $, which implies that $ \varphi_1\geqslant0  $ or $ \varphi_1\leqslant0 $ a.e. in $ \Omega $.
		
		\bigskip
		
		Next, we turn to prove \textbf{(b)}.  Suppose by contradiction that there exists another eigenfunction $ \varphi_2\in X_0^1, $ $ \varphi_1\neq \varphi_2 $, corresponding to $ \lambda_1 $.  We may assume that $\varphi_2\geqslant 0$ a.e. in $\Omega$ since it was already
			proved that all the first eigenfunctions have a sign.
			Setting $ \bar{\varphi}_2=\frac{\varphi_2}{\|\varphi_2\|_{L^2(\Omega)}} $ and $ u_1=\varphi_1-\bar{\varphi}_2 $.
		
		We claim that $ u_1=0 $ a.e. in $ \Omega $. Indeed, we observe  that $ u_1 $ is a solution of Euler-Lagrange equation~\eqref{eq:eigenvalue problem}, thus $u_1$ is also
			 the eigenfunction related to $ \lambda_1 $,  so we deduce $ u_1\leqslant0 $ or $ u_1\geqslant 0 $ a.e. in $ \Omega $, that is, $ \varphi_1\leqslant\bar{\varphi}_2  $ or $ \varphi_1\geqslant\bar{\varphi}_2  $ a.e. in $ \Omega $, which implies that $ {\varphi_1}^2\leqslant \bar{\varphi}_2^2 $ or $ \varphi_1^2\geqslant\bar{\varphi}_2^2 $. However, $ \|\varphi_1\|_{L^2(\Omega)}=1=\|\bar{\varphi}_2\|_{L^2(\Omega)} $, we get $ \varphi_1^2=\bar{\varphi}_2^2 $, hence $ u_1=0 $ a.e. in $ \Omega $. Therefore, if $ \varphi\in X_0 $ is an eigenfunction related to $ \lambda_1 $, then $ \varphi= |\varphi|_2\varphi_1 $.   If $\varphi_2\leqslant 0$, using a similar argument to the case that $\varphi_2\geqslant 0$, one infers that  $ \varphi= -|\varphi|_2\varphi_1 $. From this, we conclude that $\lambda_1$ is simple.
		
		\bigskip

		Finally, as for \textbf{(c)}  \textbf{(e)} and \textbf{(f)}, since $\varphi_1$ is a weak solution of \eqref{eq:eigenvalue problem} and satisfies the assumption \eqref{eq: H 1}, applying  Theorems~\ref{th: regularity} and \ref{th:C^{1,alpha}}, one has $\varphi_1\in C^{1,\alpha}(\overline{\Omega})$ for some $\alpha\in(0,1)$. Moreover, it follows from Theorem~\refeq{th C^2,alpha interior}  that $ \varphi_1\in C^{2,\alpha}_{\rm loc}(\Omega)\cap C^{1,\alpha}(\overline{\Omega}) $. 	As a consequence of this and \textbf{(a)}, employing the strong maximum principle, we obtain that, either~ $\varphi_1>0$ or $\varphi_1<0$.
	\end{proof}

	\end{appendix}
	}

\section*{Acknowledgments}
The authors would like to thank the anonymous referee for carefully reading the manuscript and the valuable comments to improve the paper, especially for the suggestions on the $L^\infty$-boundedness.  

\bibliographystyle{is-abbrv}

\bibliography{reference}

\end{document}